\documentclass[10pt,bezier,amstex]{article}  
\usepackage{color}              
\usepackage{graphicx}
\usepackage[]{amsmath}
\usepackage{amssymb}
\usepackage{hyperref,bbm}
\definecolor{MyDarkBlue}{rgb}{0,0.08,0.50}
\definecolor{BrickRed}{rgb}{0.65,0.08,0}
\usepackage{amsmath}

\hypersetup{
colorlinks=true,       
    linkcolor=MyDarkBlue,          
    citecolor=BrickRed,        
    filecolor=red,      
    urlcolor=cyan           
}

\usepackage{enumerate}
\usepackage[mathscr]{eucal}
\usepackage{graphicx}
\usepackage{latexsym}
\usepackage{amssymb}
\usepackage{psfrag}
\usepackage{epsfig}
\usepackage{tikz}
\usetikzlibrary{automata,positioning}

\usepackage{mhchem}

\include{psbox}

\usepackage{amsfonts}

\usepackage{graphicx}

\usepackage{amsfonts}

\usepackage{color}

\newtheorem{Theorem}{Theorem}
\newtheorem{Assumption}{Assumption}
\newtheorem{Lemma}{Lemma}[section]
\newtheorem{lemma}{Lemma}[section]
\newtheorem{Proposition}[Lemma]{Proposition}

\newtheorem{remark}[Lemma]{Remark}

\newcommand\bigCI{\mathop{\underline{\raisebox{0pt}[0pt][1pt]{$\;||\;$}}}}

\newcommand{\s}{\quad}

\newcommand{\R}{\mathbb{R}}

\newcommand{\non}{\nonumber}

\newcommand{\beq}{\begin{eqnarray*}}
\newcommand{\eeq}{\end{eqnarray*}}
\newcommand{\beqn}{\begin{eqnarray}}
\newcommand{\eeqn}{\end{eqnarray}}

\newcommand{\bt}{\begin{Theorem}}
\newcommand{\et}{\end{Theorem}}

\newcommand{\bas}{\begin{Assumption}}
\newcommand{\eas}{\end{Assumption}}

\newcommand{\be}{\begin{equation}}
\newcommand{\ee}{\end{equation}}

\newcommand{\cls}{\mathcal{S}}

\newcommand{\ci}{\perp\!\!\!\perp}

\newcommand{\mb}[1]{{\mathbf #1}}

\DeclareMathOperator*{\esssup}{ess\,sup}

\numberwithin{equation}{section}

\definecolor{darkgreen}{rgb}{0,.4,0}
\definecolor{darkagenta}{rgb}{.5,0,.5}
\definecolor{darkred}{rgb}{1,0,0}
\definecolor{darkblue}{rgb}{0,0,.4}

\begin{document}

\tikzset{every node/.style={auto}}
 \tikzset{every state/.style={rectangle, minimum size=0pt, draw=none, font=\normalsize}}
 \tikzset{bend angle=20}

	\author{Abhishek Pal Majumder
	\thanks{University of Reading}
	}

\title{Long time behavior of a stochastically modulated infinite server queue}
\maketitle




\begin{abstract}We consider an infinite server queue where the arrival and the service rates are both modulated by a stochastic environment governed by an $S$-valued stochastic process $X$ that is ergodic with a limiting measure $\pi\in \mathcal{P}(S)$. Under certain conditions when $X$ is semi-Markovian and satisfies the renewal regenerative property, long-term behavior of the total counts of people in the queue (denoted by $Y:=(Y_{t}:t\ge 0)$) becomes explicit and the limiting measure of $Y$ can be described through a well-studied affine stochastic recurrence equation (SRE) $X\stackrel{d}{=}CX+D,\,\, X\ci (C, D)$. We propose a sampling scheme from that limiting measure with explicit convergence diagnostics. Additionally, one example is presented where the stochastic environment makes the system transient, in absence of a `no-feedback' assumption.
\end{abstract}




\section{Introduction} Infinite server queues form an important building block of queueing systems (chapter $6$ of \cite{robert2013stochastic}) and stochastic models at large. In this model, customers arrive at a system equipped with an unlimited number of service stations, via a time-homogeneous Poisson process, and the service times follow an exponential distribution. Variations of this model are applied across a wide range of fields, including biological systems \cite{anderson2016functional}, \cite{Jansen2016ldp}, \cite{Palomo2023flatten}; communication systems \cite{Malhotra2001traffic}; healthcare management \cite{worthington2020infinite}; transportations \cite{baykal2009modeling2}, \cite{baykal2009modeling}, and service system management and decision support \cite{Sonenberg2024prison}.

In all of the formulations mentioned above, it is natural to consider both the arrival and service rates to be non-stationary, driven by a common background stochastic process \( X = (X_s: s \geq 0) \). The process $X$ governs the switching between different regimes or environments, with states taking values in \( S \), a set that is countable and may be infinite. Let \(\lambda(\cdot), \mu(\cdot): S \to \mathbb{R}_{\geq 0}\) be two non-negative functions. Consider an infinite server queue where customers arrive according to a Poisson process with a stochastic time-dependent intensity \((\lambda(X_s): s \geq 0)\). If $T$ is the instant when a customer enters and is instantaneously assigned to a service station, her service time will be the first event of another Poisson process (that is independent of others) with non-homogeneous intensity \((\mu(X_s):s\ge T)\) starting from $T.$ The  dynamics of the total number of customers in the system at \( t \ge 0\) is denoted by the process \( Y := (Y_t: t \geq 0)\), taking values in \(\mathbb{N}\) for each $t$. If \(\lambda\) and \(\mu\) are constants rather than functions, using birth-death process formulation it can be shown that:
\[
\mathcal{L}(Y_t) \to \text{Poisson}\left(\frac{\lambda}{\mu}\right) \quad \text{as} \quad t \to \infty.
\]
 Abusing the notations by $\text{Bin}(n,p), \text{Poi}(\lambda)$ we denote a Binomial and a Poisson random variable with parameters $(n,p)\in \mathbb{N}\times [0,1]$ and $\lambda\in\R_{\ge 0}$ respectively as well as their laws,  depending on the context. For any three random variables $X,Y,Z$ the notation $X\ci Y \mid Z$ implies that conditioned on $Z,$ $X$ is independent of $Y.$

\vspace{0.3 cm}
We assume the background process $X$ and $Y$ jointly satisfy the following property, which is crucial for the ergodicity of the joint process $(X, Y).$ 
\bas\label{A0} There is no-feedback of $Y$ on the marginal dynamics of $X,$ i.e,  $$X_{t}\ci Y_{s} \mid X_{s}\s\text{for any }\s t>s>0.$$
\eas 

Suppose $N_{1}(\cdot)$ and $\{N_{2,i}(\cdot): i\in\mathbb{N}\}$ are classes of independent unit rate Poisson processes that are independent of everything else. Denote the filtration by $\mathcal{F}^{X}_{t}:=\sigma\{X_{s}:0\le s\le t\}$ of the background stochastic environment. Using Proposition 6.2 of \cite{robert2013stochastic}, it follows that conditioned on $\mathcal{F}^{X}_{t},$ $Y_{t}$ has the same distribution as the unique solution of the following stochastic differential equation
\[dY_{t}=N_{1}\Big(\lambda(X_{t})dt\Big)- \sum_{i=1}^{Y_{t-}}N_{2,i}\Big(\mu(X_{t})dt\Big),\]
or equivalently the following stochastic integral equation
\[Y_{t}=Y_{0}+N_{1}\Big(\int_{0}^{t}\lambda(X_{s})ds\Big)- \sum_{i=1}^{\infty}\int_{0}^{t}1_{\{i\le Y_{s-}\}}N_{2,i}\Big(\mu(X_{s})ds\Big).\]

The stochastic integral equation for $Y_{t}$ above, can be expressed further as the distributional solution of 
\beqn
Y_{t}&\stackrel{d}{=}&Y_{0}+ N_{1}\Big(\int_{0}^{t}\lambda(X_{s})ds\Big)-N_{2}\Big(\int_{0}^{t}Y_{s-}\mu(X_{s})ds\Big),\s\text{for each}\,\, t>0,\label{inftysq}
\eeqn
where $Y_{0}$ denotes the initial number of people in the system at $t=0,$ and $N_{2}$ is a unit rate Poisson process.

 We consider an example here. The arrival rate $\lambda(\cdot)=1,$ and the service rate $\mu(X_{s})=X_{s},$ where the  background environment $X=(X_{s}:s\ge 0)$ is driven by a $S:=\{0,1\}$-valued continuous-time Markov chain with switching rates are given by $Q=\begin{pmatrix} -\lambda_{01} & \lambda_{01}\\ \lambda_{10} & -\lambda_{10}
\end{pmatrix}$ with $\lambda_{10}=1000\lambda_{01},\,\,\&\,\, \lambda_{10}=1.$ Figure \ref{ex1} exhibits one sample path of $Y.$ 

\begin{figure}[h]
\centering
\includegraphics[width=12cm, height=7.5cm]{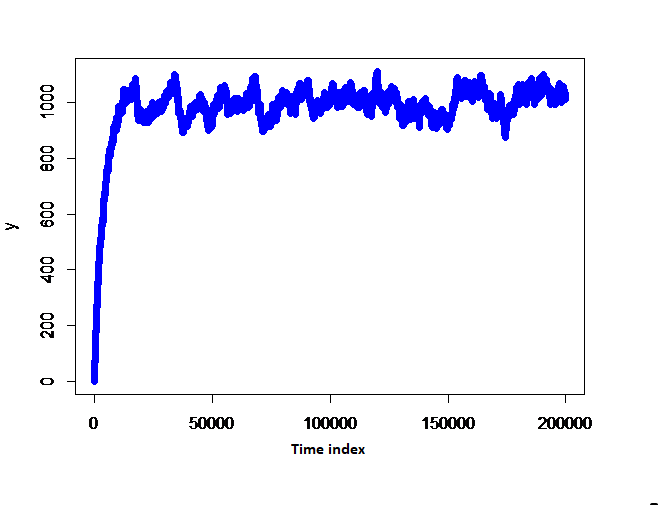}
\caption{A sample trajectory of \eqref{inftysq} where  $X$ is a $\{0,1\}$-valued CTMC with switching rates $\lambda_{01},\lambda_{10}$ where $\lambda_{10}=1000\lambda_{01}$, and $\lambda(\cdot)=1, \mu(X_{s})=X_{s}$ for any $s\ge 0.$
}\label{ex1}
\end{figure}

The dynamics of $X$ implies that the long-time proportion of occurrence of the state $X=1,$ is $\frac{1}{1001},$ which is minuscule, causing the counts of $Y$ initially rise because the event $\{Y \,\, \text{reduces by one}\}$ took place only when $\{X=1\}$ (i.e, once in $1001$ time unit (in average)), while in the remaining time $Y$ increases by a rate $1$ Poisson process. Hence the count $Y$ grows almost linearly. When the count approaches $1000,$ then the number of incoming people in the whole regenerating cycle of $X$ denoted by $\{0\to 1, 1\to 0\}$ will be in a stochastic equilibrium with the total number of counts reduced while $X=1,$ in that cycle. This is due to the aggregate large counts of $Y,$ the intensity of the number of services done by the system governed by $N_{2}\Big(\int_{0}^{t}Y_{s-}\mu(X_{s})ds\Big)$ in \eqref{inftysq} will be large, which matches approximately with the total counts of incoming people in every regeneration cycle. Hence, in the long run, the count of $Y$ will be in some sort of equilibrium, though that limiting measure is not known in a closed form, it would be very challenging to characterize that in a tractable form if $S$ is countably infinite. 

Some notable works on Markov-modulated infinite server queues are done in \cite{blom2014markov},\cite{anderson2016functional} that focus on establishing fluid limit and pathwise approximation via FCLT in mind respectively. In \cite{cinneide1986dynamics} the authors considered this problem when the environment is modeled by a finite state-space Markov process and derived factorial moments for the steady state. But, for an infinite state space, this approach may not yield a closed-form expression. In this article, we will give a representation for an $S$-valued semi-Markovian $X$ that is more general than a continuous time-Markov chain, with some favorable regenerative renewal property, discussed in the next section. In a Markov-modulated environment, the sojourn time for any state $\{X_{t}=i\}$ follows an Exponential distribution, with a rate that depends only on the current state. In contrast, in a semi-Markovian environment, sojourn times can follow any distribution, not limited to the Exponential, and may depend not only on the current state but also on the next state. This added flexibility in semi-Markov modulation allows for more versatile modeling, where the distribution of sojourn times can be tailored to reflect transitions influenced by future state. This dependence on future state makes semi-Markov models particularly suitable for incorporating control-related formulations into the latent dynamics of stochastic models. The main contributions of this work are :
\begin{enumerate}[(a)]
\item Given $\lambda,\mu$ being non-constant functions, a mixture of Poisson distribution emerges in the time asymptotic limit where the mixing measure can be explicitly characterized as a unique solution of a stochastic recurrence equation.
\item Due to the explicit characterization of the mixing measure, a representation of the moments of the limiting measure is given, along with an approximate sampling scheme with the relevant convergence diagnostics.
\end{enumerate}

While finishing this manuscript, we came across a related work \cite{d2008m} that closely aligns with the objective of this study. However, our approach differs fundamentally, as it builds on the renewal regenerative structure of the background semi-Markov process, in contrast to the matrix-geometric method used in \cite{d2008m}. This distinction enables us to express the time-asymptotic measure of $Y$ explicitly, even when $S$ is countably infinite. Additionally, the state-space representation of the mixing measure helps us create an easy-to-use sampling scheme to simulate from it, which can be used for analyzing its statistical properties or making long-term predictions via making an effective confidence interval that is exact upto a certain level, asymptotically in time.

This paper is organized as follows. Section \ref{sec2} sets the structure of the semi-Markovian environment $X$ along with preliminary assumptions. In Section \ref{MR}, exact long-time characterization for $(Y,X)$ is given along with the recursive expression for the moments of the marginal limiting measure of $Y$. Section \ref{SS} presents an approximate simulation scheme to generate a sample from the mixture measure, and hence from the steady-state distribution, along with some examples.   In Section \ref{example} we present an example where $Y$ marginally behaves as a marginal of an infinite server queue, but the environment $X$ gets feedback from $Y,$ and hence Assumption \ref{A0} is violated, resulting in long-time transient behaviour for $Y$. In Section \ref{Conclu} we present some comments and future research directions. All the proofs of the results will be presented in Section \ref{proof} with two additional results in the Appendix \ref{App}.

\section{Preliminaries and description of the semi-Markovian $X$}\label{sec2}

As a generalization of the continuous-time Markov chain, the theoretical foundation of the Semi-Markov process and its time asymptotic behaviors were first explored by Cinlar in his seminal works \cite{cinlar1969markov} \& \cite{ccinlar1969semi} along with his earlier studies on queues \cite{ccinlar1967queues}. We follow a similar formalism here. Let $X$ be an $S$-valued semi-Markov process (where $S$ is a countable set) with some underlying structures and relevant assumptions that we will introduce here. Consider a time-homogeneous Markov chain $(U, T) = (U_n, T_n)_{n \in \mathbb{N}}$ taking values in $S \times \mathbb{R}_{+}$, with transition dynamics governed by the kernel
\beqn
Q_{ij}(t) := P\big[U_{n+1} = j, T_{n+1} \leq t \mid U_n = i, T_n = s\big], \quad \forall n \in \mathbb{N}, \, \forall s, t \geq 0,
\label{Q}
\eeqn
where $P[U_0 = i, T_0 = 0] = p_i$ for each $i \in S$, and $\sum_{i \in S} p_i = 1$. The kernel $Q_{ij}(t)$, which represents the transition probabilities between states, does not depend on the specific value of $T_n = s$. Therefore, we can simplify the expression for the transition kernel as
\[
Q_{ij}(t) = P\big[U_{n+1} = j, T_{n+1} \leq t \mid U_n = i\big].
\]

Let the cumulative sojourn time be denoted by $S_n := \sum_{i=0}^{n} T_i$ for $n \in \mathbb{N}$. The process $X = (X_t : t \geq 0)$, that captures the regime-switching dynamics, is then defined as
\beqn
X_t := U_{N_t}, \quad \text{where} \quad N_t := \esssup\Big\{n : S_n \leq t\Big\}.
\label{Ydef}
\eeqn
This definition indicates that $T_n$ represents the sojourn time of the process $X$ in state $U_{n-1}$ for $n \geq 1$, with $T_0 := 0$.

Marginally, the discrete-time process $U = (U_n : n \in \mathbb{N})$ forms an $S$-valued Markov chain (independent of everything else) with transition kernel $P = (P_{ij}: i,j \in S)$, where
\[
P_{ij} := P\big[U_{n+1} = j \mid U_n = i\big], \quad P_{ii} = 0 \quad \forall i \in S.
\]
This defines the marginal dynamics of the embedded chain $U$. The process $X$ can thus be seen as the marginal of a Markov renewal process $(S_n, U_n)$, commonly known as a semi-Markov process \cite{cinlar1969markov}. Notably, $X$ is not necessarily Markovian.

Next, introduce the set
\beqn
\mathcal{G} := \Big\{F_{ij} \in \mathcal{P}(\mathbb{R}_+) : i,j \in S, \, i \neq j, \, P_{ij} > 0, \, F_{ij}(0) = 0, \, F_{ij}(\infty) = 1\Big\},
\label{G}
\eeqn
which consists of probability distributions (without atoms at 0) for the sojourn times of $X$. Specifically, $F_{ij}$ represents the law of the sojourn time at state $i$ when transitioning to state $j \in S$. For all $n \geq 0$, we define
\[
F_{ij}(t) := P\big[T_{n+1} \leq t \mid U_n = i, U_{n+1} = j\big].
\]
Using these definitions, we can decompose $Q_{ij}(t)$ as follows:
\beqn
Q_{ij}(t) = P\big[U_{n+1} = j \mid U_n = i\big] P\big[T_{n+1} \leq t \mid U_n = i, U_{n+1} = j\big] = P_{ij} F_{ij}(t).
\label{multiplic}
\eeqn
Given the transition kernel $P$, the set $\mathcal{G}$, and the initial distribution $\{p_i : i \in S\}$, the dynamics of the pair $(U, T)$ are fully specified. Consequently, the process $X$ defined in \eqref{Ydef} is well-defined.

Note that the semi-Markov process $X$ allows for a more complex dependence structure compared to a continuous-time Markov chain, which can be recovered as a special case where the sojourn time distributions are exponential, i.e., $F_{ij}(t) = 1 - e^{-\lambda_i t}$ for some $\lambda_i > 0$ and all $i,j \in S$. The semi-Markov formulation gives more flexibility and control on the sojourn time distributions which may depend on the next state. This opens up a huge number of modeling opportunities for the underlying regime-switching mechanisms.  

Finally, the semi-Markov process $X$ is called conservative at state $j \in S$ if
\[
P[N_t < \infty \mid X_0 = j] = 1 \quad \forall t \in [0, \infty),
\]
which is a regularity condition (see \cite{cinlar1969markov}, p. 150). This condition ensures that $X$ is non-explosive, meaning that the number of transitions within any finite time interval is almost surely finite when starting from any state $j \in S$.

\bas\label{As0}Some assumptions on the dynamics of $(U,T)$ are below. 
\begin{enumerate}[(a)]
    \item The conditional probability 
    \[
    P\big[U_{n+1} = j, T_{n+1} \leq t \mid U_n = i, T_n = s\big]
    \]
    does not depend on \(s\). Consequently, \(s\) is absent from the notation \(Q_{ij}(t)\) in the definition \eqref{Q}.

    \item The embedded chain \(U\) is a time-homogeneous, irreducible, aperiodic, positive recurrent, discrete-time homogeneous Markov chain, taking values in \(S\). It is independent of everything else. The transition matrix \(P\) satisfies the following properties: 
    \[
    P_{ii} = 0, \quad P_{ij} \in [0,1], \quad \sum_{j \in S} P_{ij} = 1 \quad \forall i,j \in S.
    \]
    Furthermore, the chain has a unique stationary distribution \(\mu \in \mathcal{P}(S)\), which satisfies the measure-valued equation \(\mu = \mu P\).

    \item The set \(\mathcal{G}\) is such that for all \(i,j \in S\), the distribution functions \(F_{ij}\) satisfy \(F_{ij}(0) = 0\) and contain no atoms in \(\mathbb{R}_{\geq 0}\), meaning each \(F_{ij}\) is absolutely continuous with respect to the Lebesgue measure on \(\mathbb{R}_{\geq 0}\). For each \(i,j \in S\), define the mean of the random variables distributed according to \(F_{ij}(\cdot)\) as
    \[
    m_{ij} := \int_{0}^{\infty} \big(1 - F_{ij}(t)\big) dt,
    \]
    and assume that 
    \[
    \sup_{i,j \in S} m_{ij} < \infty.
    \]

    \item The semi-Markov process \(X\) is conservative for all \(j \in S\).
\end{enumerate}
\eas

For any $i,j\in S,$ let $W_{ij}$ be defined as the $F_{ij}$ distributed random variable denoting the sojourn time of $X$ at state $i$ conditioned on the future state $j$. Then we define the distribution of sojourn time at state $i$ unconditional on the future value as
$$P[T_{n+1}<t| U_{n}=i]= \sum_{j\in S}P_{ij}F_{ij}(t)=:F_{i}(t)$$ denoted by $F_{i}(\cdot)$ and denote the corresponding random variable by $W_{i}.$ For each $i\in S,$ by $m_{i}$ we denote the mean of $W_{i}$ as
$$m_{i}:=EW_{i}=\int_{0}^{\infty}(1-F_{i}(t))dt=\int_{0}^{\infty}(1-\sum_{j\in S}P_{ij}F_{ij}(t))dt=\sum_{j\in S}P_{ij}m_{ij}.$$

Assumption \ref{As0}(a) implies that for each $j \in S$, we have $\mu_j > 0$, and the positive recurrence of $U$ guarantees the existence of hitting times at state $j$, defined as:
$$
N_{0}^{j} := \inf\{n \ge 0 : U_{n} = j\}, \quad N_{k}^{j} := \inf\{n > N_{k-1}^{j} : U_{n} = j\} \quad \text{for} \,\, k \geq 1.
$$
Denote the set $\{N_{k-1}^{j}, N_{k-1}^{j} + 1, \dots, N_{k}^{j} - 1\}$, which is a successive array of integers, by $A_{k}^{j}$, representing the $k$-th return of the process $U$ to state $j$ for $k \geq 1$. The positive recurrence of $U$ implies that $E|A_{k}^{j}| = E[N_{k}^{j} - N_{k-1}^{j}] < \infty$. A transition from state $i$ to state $j$ (denoted as $(i, j)$ or $i \to j$) for the chain $U$ occurs if $P_{ij} > 0$, and the set of all possible transitions is given by $\{(i, j) \in S \times S : P_{ij} > 0\}$.

The recursion path $A_{k}^{j}$ can alternatively be represented as the set $A_{k}^{j} = \{j, i_{1}, i_{2}, \dots, i_{l_{1}}\}$, where $l_{1} = [N_{k}^{j} - N_{k-1}^{j}] - 1$, corresponding to the sequence of transitions $j \to i_{1} \to i_{2} \to \dots \to i_{l_{1}} \to j$. Let $\widetilde{A}_{k}^{j}$ denote the set of tuples representing the successive state transitions in $A_{k}^{j}$, expressed as:
\[
\widetilde{A}_{k}^{j} = \{(j, i_{1}), (i_{1}, i_{2}), (i_{2}, i_{3}), \dots, (i_{|A_{k}^{j}|-1}, j) : i_{l} \in S, \, i_{l} \neq j \,\, \forall l = 1, \dots, |A_{k}^{j}| - 1\}.
\]
Note that $\widetilde{A}_{k}^{j}$ corresponds one-to-one with $A_{k}^{j}$, as both describe the $k$-th recursion cycle from state $j$ back to itself. It is possible for a specific transition $(i_{k}, i_{k+1})$ to appear multiple times in $\widetilde{A}_{k}^{j}$, as it may recur several times within a single cycle of returns to state $j$.

Observe that $P[T_{0}=0]=1$ and for any arbitrary $u_{1},\ldots,u_{n}>0$ and any arbitrary $i_{1},\ldots,i_{n+1}\in S$ one has the following identity for any $n\in \mathbb{N}$
\beqn
P\big[T_{1}\le u_{1},\ldots,T_{n}\le u_{n} |U_{0}=i_{1},U_{1}=i_{2},\ldots, U_{n}=i_{n+1}\big]&\stackrel{(a)}{=}&\frac{\prod_{k=1}^{n}Q_{_{i_{k}i_{k+1}}}(u_{k})}{\prod_{k=1}^{n} P_{i_{k}i_{k+1}}}\s\non\\
&\stackrel{(b)}{=}&\prod_{k=1}^{n}F_{_{i_{k}i_{k+1}}}(u_{k})\label{Prod}
\eeqn
where $(a)$ and $(b)$ in above equalities hold due to \eqref{Q} and \eqref{multiplic} (as a result of Assumption \ref{As0} (c)) respectively. Display \eqref{Prod} suggests that conditioned on path $\{U_{0}=i_{1},U_{1}=i_{2},\ldots,U_{n}=i_{n+1}\}$ the sojourn times $\{T_{1},\ldots T_{n}\}$ in corresponding states are independent, which also prompts us to consider the regenerating renewal intervals for $X$ as well. We define the hitting times of $X$ at state $j\in S,$ denoted by $\{\tau_{k}^{j}:k\ge 0\}$ as
$$\tau_{0}^{j}:=\sum_{i=0}^{N_{0}^{j}}T_{i},\s \tau_{k}^{j}:=\sum_{i=0}^{N_{k}^{j}}T_{i}\s\forall k\ge 1,$$
and denote the last hitting time to state $j$ before $t$ by 
\beqn
g_{t}^{j}:=\max(\sup\{n: \tau_{n}^{j}\le t\},0),\s \sup\emptyset=-\infty.\label{g_{t}}
\eeqn

  It is clear that any functionals of $X$ in $\{[\tau_{i}^{j},\tau_{i+1}^{j}):i\ge 0\}$ behave identically and independently as $\Big\{\sum_{i=N_{k-1}^{j}+1}^{N_{k}^{j}}T_{i}:k\ge 1\Big\}$ are iid due to the fact that $\mathcal{G}$ is a fixed set and $\{A_{k}^{j}\}_{k\ge 1}$ are regenerating sets for $U$ at state $j\in S.$ For $Y$ denote the $k$-th recursion interval at state $j\in S$ by $I_{k}^{j}:=[\tau_{k-1}^{j},\tau_{k}^{j}).$  From Assumption \ref{As0}(b) it follows that
\beqn
E|I_{k}^{j}|&=&E\sum_{n=1}^{\infty}T_{n} 1_{\{(U_{n-1},U_{n})\in \widetilde{A}_{k}^{j}\}}\non\\
&=&E \Big[\sum_{n=1}^{\infty}E\big[T_{n}| U_{n}=i_{n},U_{n-1}=i_{n-1}\big]1_{\{(U_{n-1},U_{n})\in \widetilde{A}_{k}^{j}\}}\Big]\non\\
&\le& \big(\sup_{i,j\in S}m_{ij}\big) E|\widetilde{A}_{k}^{j}|\non\\&=&\big(\sup_{i,j\in S}m_{ij}\big) E|A_{k}^{j}|<\infty,\s\forall j\in S,\s \forall k\in \mathbb{N},\non
\eeqn
as a consequence of $E\big[T_{n}| U_{n}=i_{n},U_{n-1}=i_{n-1}\big]=m_{i_{n-1},i_{n}}\le \sup_{i,j\in S}m_{ij}$, uniformly for all $(i_{n-1},i_{n})\in S^{2}.$

 For any event $A$ by $P_{_{i}}[A]$ we denote $P[A\mid Y_{0}=i].$ For any random variable $Z,$ distribution of $Z$ is often denoted as $\mathcal{L}(Z).$ For any functions $c,d:S\to\R$ and $j\in S$, define 
\beqn
\,\,\,\,\, G^{c,d}_{j}(x):=\int_{0}^{x}d(j)e^{-c(j)(x-s)}ds=x d(j)1_{\{c(j)=0\}}+ \frac{d(j)}{c(j)}\Big(1-e^{-xc(j)}\Big)1_{\{c(j)\neq 0\}}.
\label{Gfun}
\eeqn

\begin{remark}\label{R0}
Assumption \ref{As0}(a) is crucial for running the regenerative renewal argument for $X$. Consequently, any functionals of $X$ within the intervals ${[\tau_{i}^{j}, \tau_{i+1}^{j}): i \ge 0}$ behave identically and independently for each $j \in S$. If Assumption \ref{As0}(c) does not hold, the process $X$ would not independently regenerate after hitting state $j$, as the sojourn time at state $j$ would depend on the sojourn time of the previous state before jumping to $j$. This dependency contradicts the regeneration property of $X$ within the intervals ${[\tau_{i}^{j}, \tau_{i+1}^{j}): i \ge 0}$.
\end{remark}

Theorem 7.14 of \cite{cinlar1969markov} (Page 160) suggests that if $X$ is an irreducible, recurrent (The term `Persistent' is used in \cite{cinlar1969markov} for recurrent), aperiodic semi-Markov process along with $m_{i}:=\sum_{j\in S}P_{ij}m_{ij}<\infty,$ then
\beqn
\lim_{t\to\infty}P_{i}[X_{t}=j]=\frac{\mu_{j}m_{j}}{\sum_{k\in S}\mu_{k}m_{k}},\s\forall i\in S\s\label{lim1}
\eeqn
which follows from conditions of Assumption \ref{As0}. By $\pi$ denote the probability vector $(\pi_{i}:i\in S)$ such that 
\beqn
\pi_{j}:= \frac{\mu_{j}m_{j}}{\sum_{k\in S}\mu_{k}m_{k}}\,\,\forall j\in S.\label{pi}
\eeqn
 From Assumption \ref{As0} it follows that $\pi_{j}>0, \forall j\in S.$ We loosely describe a process to be ergodic when it converges to a limiting distribution that does not depend on the initial distribution and hence a semi-Markov process $X$ having all aforementioned properties (irreducible, recurrent, aperiodic, $m_{k}<\infty,\,\,\forall k\in S$) is ergodic.

Regenerative property of $X$ suggests that any specific functional of $X$ in $\{[\tau_{i-1}^{j},\tau_{i}^{j}): i\ge 1\}$ behave identically independent random variables. Denote the sigma algebra $\sigma\{X_{t}: t\in [\tau_{i-1}^{j},\tau_{i}^{j})\}$ by $\mathcal{H}_{i}$ for each $i\ge 1.$ In this setting we have following result.

For any time $t>0,$ denote $t-\tau^{j}_{g_{t}^{j}}, \tau_{g_{t}^{j}+1}^{j}-t$ by $A_{j}(t),B_{j}(t)$ as respectively backward residual time and forward residual times at state $j\in S$. Clearly $A_{j}(t)+B_{j}(t)= \tau_{g_{t}^{j}+1}^{j}-\tau^{j}_{g_{t}^{j}},$  length of the regenerating interval containing $t.$ Results from \cite{cinlar1969markov},\cite{asmussen2008applied}  suggest if $E[\tau^{j}_{2}-\tau_{1}^{j}]<\infty,$ (which follows from $E|I_{1}^{j}|<\infty,$ ensured under Assumption \ref{As0}) then both $A_{j}(t),B_{j}(t)$ are $O_{_{P}}(1)$ as both quantities
\beqn
P[A_{j}(t)>x]\to \frac{\int_{x}^{\infty}P\big[|I_{k}^{j}|>y\big]dy}{E|I_{k}^{j}|},\s P[B_{j}(t)>x]\to\frac{\int_{x}^{\infty}P\big[|I_{k}^{j}|>y\big]dy}{E|I_{k}^{j}|},\label{eRes}
\eeqn
as $t\to\infty$.

 The following proposition is a consequence of the regenerative renewal property of $X.$
\begin{Proposition}\label{Psemi1}
Suppose Assumption \ref{As0} holds. For any $t>\tau_{0}^{j},$
\begin{enumerate}[(a)]
\item the distribution of $\tau_{g^{j}_{t}+2}^{j}-\tau_{g^{j}_{t}+1}^{j}$ is identical as $\tau^{j}_{1}-\tau_{0}^{j}.$
\item any functional of $X$ in $\{s: s\ge \tau^{j}_{g_{t}^{j}+1}\}$ is independent of $g_{t}^{j}$ and identically distrubuted as the same functional of $X$ on $\{s:s\ge \tau_{0}^{j}\},$ for any $t\ge \tau_{0}^{j}.$
\item suppose conditioned on any event $K_{t},$ marginals of two processes $X^{(1)}:=\big(X^{(1)}_{t}:t\ge 0\big), X^{(2)}:=\big(X^{(2)}_t:t\ge 0\big)$ at time $t$ are respectively $\sigma\{\mathcal{H}_{i}:i\le g^{j}_{t}\}$ and $\sigma\{\mathcal{H}_{i}: i\ge g^{j}_{t}+2\}$ measurable for any $j\in S$. Then conditioned on the event $\{X_{t}=j'\}$ following holds 
$$P[X^{(1)}_{t}\in A ,X^{(2)}_{t} \in B\mid K_{t},X_{t}=j']=P[X^{(1)}_{t}\in A \mid K_{t},X_{t}=j'] P[X^{(2)}_{t} \in B\mid  K_{t}],$$
for any two sets $A,B$ for any $j'\in S.$
\end{enumerate}
\end{Proposition}

For each $j\in S,$ let $\pi^{*}_{j}$ be the distribution on $\R^{+}$ such that for any $x>0,$ 
\beqn
\pi^{*}_{j}[x,\infty):=\frac{\sum_{k\in S}P_{jk}\int_{x}^{\infty}(1-F_{jk}(y))dy}{m_{j}}.\label{pi_j^*}
\eeqn

\begin{remark}\label{R1}
$\pi_{j}^{*}$ in \eqref{pi_j^*} represents the size-biased distribution of averaged sojourn time at a state $j\in S$. Notably, when $F_{jk}(x)=1-e^{-\lambda_{j}x}\,\,\forall k,$ (corresponds to the case of continuous time Markov chain (CTMC)) then $\pi^{*}_{j}[x,\infty)=e^{-\lambda_{j}x}$ for $x>0.$ This holds because the Exponential distribution is the unique distributional fixed point of the size-bias transformation. Now, consider a variation where $F_{jk}(x):=1-e^{-\lambda_{jk}x},$ for some $\lambda_{jk}\in(0,\infty)\,\,\forall j,k\in S.$ In this case, $\pi^{*}_{j}$ becomes following $$\pi^{*}_{j}[x,\infty):=\frac{\sum_{k\in S}\frac{P_{jk}}{\lambda_{jk}}e^{-\lambda_{jk}x}}{\sum_{k\in S}\frac{P_{jk}}{\lambda_{jk}}}\s\forall x>0,$$
indicating $\pi_{j}^{*}$ as a mixture of Exponentials. If we assume $F_{jk}(\cdot):= F_{j}(\cdot)$ for some distribution function $F_{j}(\cdot)$ for all $j,k\in S,$  then $\pi^{*}_{j}[x,\infty)=\frac{\int_{x}^{\infty}(1-F_{j}(y))dy}{m_{j}},$ for any $x>0.$
\end{remark}

For any two measures $\nu_1, \nu_2$ that have finite first moment, their distance can be measured by Wasserstein-$1$ distance
$$\mathcal{W}_{1}(\nu_1, \nu_2)=\inf_{(Y_1, Y_2)\colon \mathcal{L}(Y_1)=\nu_1,\mathcal{L}(Y_2)=\nu_2}E\big(\|Y_1-Y_2\|_1\big),$$
that measures the L$_{1}$ distance of two given measures. 

For any $x>0,\,\,\, (x)_{k}$ denotes the $k$-th order factorial moment $$(x)_{k}:=x(x-1)\cdots (x-k+1), \s (x)_{0}=1.$$ For any two integers $k,n$ such that $0\le k \le n,$ the Stirling number of type $2$ denoted by ${n\brace k},$ are defined as  respective coefficients $\{{n\brace k}:k=1,\ldots,n\}$ of the factorial functions of order $\{(x)_{k}:k=1,\ldots,n\}$ in the expansion of $x^{n}.$ That is, for any $x>0,$ and any integer $n$ 
\beqn
x^{n}:=\sum_{k=0}^{n}{n\brace k}(x)_{k}.\label{stirling2}
\eeqn

For a given bivariate random variable $(C,D)$, the following time series is referred to as a stochastic recurrence equation (in short SRE, also referred to as random coefficient AR$(1)$) 
\begin{align*}
Z_{n+1}=C_{n+1}Z_{n}+D_{n+1}\quad\text{with}\quad (C_{i},D_{i})\stackrel{\text{i.i.d}}{\sim} \mathcal{L}(C,D),\quad Z_{n}\ci (C_{n+1},D_{n+1}) 
\end{align*} 
for an arbitrary initial $Z_{0}=z_{0}\in \R$. Let $\log^{+}(a):=\log(\max(a,1))$ for any $a>0$.  
If 
\beqn
P[C=0]=0, \quad E\log|C| <0,\s\text{and}\s E\log^{+}|D|<\infty, \label{sre_conv_conds}
\eeqn
then $(Z_{n})$ has a unique causal ergodic strictly stationary solution solving the following fixed-point equation in law: 
\beqn
Z\mathop{=}^{d}CZ+D\s\text{ with} \s Z\ci (C,D).\label{x=ax+b}
\eeqn
The condition $P[Cx+D=x]<1,$ for all $x\in \mathbb{R}$, rules out degenerate solutions $Z=x$ a.s. We refer to Corollary 2.1.2 and Theorem 2.1.3 in Buraczewski et al. \cite{buraczewski2016stochastic} for further details. For multivariate $d$-dimensional case the existence and uniqueness \cite{erhardsson2014conditions} of the distributional solution \eqref{x=ax+b} holds if
$$\prod_{i=1}^{n}C_{i}\,\,\stackrel{a.s}{\to}\,\, 0,\s\&\s \bigg(\prod_{i=1}^{n}C_{i}\bigg) D_{n+1}\,\,\stackrel{a.s}{\to}\,\, 0.$$

Let $\{R_{n}\}_{n \ge 1}$ be an arbitrary sequence of random variables, and let $N^{*}_{t}$ be an $\mathbb{N}$-valued stopped random time. Suppose there exists a random variable $R_{\infty}$ such that $R_{n} \stackrel{d}{\to} R_{\infty}$ as $n \to \infty$. Define $N^{}_{t}$ as any stopped random time (defined on the same probability space as $\{R_{n}\}_{n \ge 1}$) for which there exists an increasing function $c(\cdot)$ such that $c(t) \to \infty$ and $\frac{N^{}_{t}}{c(t)} \stackrel{P}{\to} 1$ as $t \to \infty$. In this context, Anscombe's contiguity condition (Gut \cite{gut2009stopped}, p. 16) is useful for establishing the weak convergence of the process $(R_{N^{*}_{t}} : t \ge 0)$ as $t \to \infty$. It is stated as follows:

Given $\epsilon > 0$ and $\eta > 0$, there exist $\delta > 0$ and $n_{0}$ such that
\beqn
P\bigg(\max_{\{m: |m-n|<n\delta\}}|R_{m}-R_{n}|>\epsilon\bigg)<\eta,\s\forall\s n> n_{0}.\label{Anscombe}
\eeqn 
If the sequence $\{R_{n}\}_{n \ge 1}$ satisfies \eqref{Anscombe}, then it is sufficient to conclude that $$R_{N^{*}_{t}}\stackrel{d}{\to}R_{\infty}\s\text{as}\s t\to\infty.$$

\section{Main results}\label{MR}
\bt\label{T1}
Suppose Assumptions \ref{A0} \& \ref{As0} hold for $(X,Y)$ satisfying \eqref{inftysq} along with the condition 
\beqn
E_{\pi}\mu(\cdot)>0\s\,\, \& \s\,\, E\Big[\log^{+}\Big(\int_{\tau^{j}_{0}}^{\tau^{j}_{1}}\lambda^{}(X_{s})e^{-\int_{s}^{\tau^{j}_{1}}\mu(X_{r})dr}ds\Big)\Big]<\infty\label{condT1}
\eeqn
for every $j\in S.$ The limiting law of $(Y,X)$ can be expressed as
\begin{equation}
(Y_{t},X_{t}) \stackrel{d}{\to} \bigg(\sum_{j\in S}^{}\delta_{U_{*}}(\{j\})\text{Poi}(W_{j}),U_{*}\bigg)\s \text{as} \s t\to\infty\label{ePoi1}
\end{equation}
where for each $j\in S,$  $U_{*}\ci (W_j)_{j\in S}$, $U_{*}\sim \pi$,   and $W_{j}$ is a random variable defined as
\beqn
W_j\stackrel{d}{=}
\lambda(j)\int_0^{T_j}e^{-\mu(j)(T_j-s)}ds+e^{-\mu(j)T_j}V^{*}_j,\label{Z_{j}}\label{mixture}
\eeqn

where $T_j\sim \pi_{j}^{*}$ is independent of $V^{*}_j$, and 
$\mathcal{L}(V^{*}_j)$ is the unique solution to \eqref{x=ax+b} with $(C,D)$ having the distribution of 
\begin{eqnarray}
\Big(e^{-\int_{\tau^{j}_{0}}^{\tau_{1}^{j}}\mu(X_{s})ds},
\int_{\tau^{j}_{0}}^{\tau_{1}^{j}}\lambda^{}(X_{s})e^{-\int_{s}^{\tau^{j}_{1}}\mu(X_{r})dr}ds\Big).\label{AB}
\end{eqnarray}
\et

\begin{remark}\label{R2} 
\begin{enumerate}[(a)]
\item Since $\mu(\cdot)$ is non-negative, the condition $E_{\pi}\mu(\cdot)>0$ is same as saying $P\big[\int_{\tau_{0}^{j}}^{\tau_{1}^{j}}\mu(X_{s})ds>0\big]>0$ where in general $\int_{\tau_{0}^{j}}^{\tau_{1}^{j}}\mu(X_{s})ds\ge 0$ a.s.
\item Since $E\Big[\int_{\tau^{j}_{n}}^{\tau_{n+1}^{j}}\mu(X_{s})ds\Big]=\big(E_{\pi}\mu(\cdot) \big)E|\tau_{n+1}^{j} -\tau_{n}^{j}|$ (by renewal theory), the conditions in \eqref{condT1} corresponds, in the current setting, to the general condition \eqref{sre_conv_conds} for the existence of a stationary solution to the stochastic recurrence equation
$$
V_{j,n+1}=e^{-\int_{\tau^{j}_{n}}^{\tau^{j}_{n+1}}\mu(X_{s})ds}V_{j,n}+\int_{\tau^{j}_{n}}^{\tau^{j}_{n+1}}\lambda^{}(X_{s})e^{-\int_{s}^{\tau^{j}_{n+1}}\mu(X_{r})dr}ds
$$ 
with affine invariant solution of the form 
$$
V^{*}_{j}\stackrel{d}{=}e^{-\int_{\tau^{j}_{0}}^{\tau^{j}_{1}}\mu(X_{s})ds}V^{*}_{j}+\int_{\tau^{j}_{0}}^{\tau^{j}_{1}}\lambda^{}(X_{s})e^{-\int_{s}^{\tau^{j}_{1}}\mu(X_{r})dr}ds.
$$ 
Second condition in \eqref{condT1} defines the integrability criteria for $\lambda(\cdot).$ Observe that if $\sup_{j\in S} \mu(j)<\infty$ and $E_{\pi}\big[\lambda^{}(\cdot)\big]<\infty$, then the second condition in \eqref{condT1} follows immediately as a consequence of the inequalities $\log^{+}|xy|\le \log^{+}|x|+\log^{+}|y|$ and $\log^{+}|x|\le |x|$ for any $x,y$.
\end{enumerate}
\end{remark}

A multivariate version of this result under the Markovian regime-switching context is explored in \cite{cappelletti2021dynamics} with applications in stochastic mono-molecular biochemical reaction networks.

Any integer order of moments can be computed for the time-limiting measure of the total counts which is the first marginal of distribution of the RHS of \eqref{ePoi1},

\beqn
\widetilde{\pi}_{1}:=\mathcal{L}\bigg(\sum_{j\in S}^{}\delta_{U_{*}}(\{j\})\text{Poi}(W_{j})\bigg),\s U_{*}\sim \pi,\label{pi1}
\eeqn
 given some integrability conditions \eqref{condT1}. 

Denote the random quantities in \eqref{AB} by $(C_{1}^j,D_{1}^j).$ The superscript in $(C_{1}^j,D_{1}^j)$ should not be confused with idea of exponents. It is there to represent the fact that these random elements are outputs of the first recursion of $X$ at state $j\in S.$ For any $\alpha\in\R,$ by $\big((C_{1}^j)^{\alpha},(D_{1}^j)^{\alpha}\big),$ we denote $\alpha$-th exponent of $(C_{1}^j,D_{1}^j)$.
 
\begin{Proposition}\label{PrMom}Suppose assumptions and conditions of Theorem \ref{T1} hold. There exists an integer $n\in \mathbb{N}$ such that 
\beqn
E\big[(C_{1}^j)^{l}\big]<1,\s \text{and}\s E[(C_{1}^j)^{k}(D_{1}^j)^{l-k}]<\infty\s\text{ for all }\,\,\,\,0\le k\le l\le n.\label{mom0}
\eeqn
 Then
\beqn
\int_{0}^{\infty}y^{n}\widetilde{\pi}_{1}(dy)= \sum_{j\in S}\pi_{j}\sum_{k=1}^{n}{n\brace k}E\Big[\lambda(j)\int_0^{T_j}e^{-\mu(j)(T_{j}-s)}ds+e^{-\mu(j)T_j}V^{*}_j\Big]^{k},\label{mom1}
\eeqn

where ${n\brace k}$ is the Stirling number of type $2,$ for all $j\in S,$ $T_j\sim \pi_{j}^{*},$ $T_j\perp V^{*}_j$ and for all $1\le l\le n,$ \s $m^{(l)}_{j}:=E\big(V^{*}_j\big)^{l}$ can be found out from the following recursion with $m^{(0)}_{j}=1,$ and 
\beqn
m^{(l)}_{j} = \frac{1}{1-E[(C_{1}^j)^{l}]}\sum_{k=0}^{l-1}{l\choose k}E\big[(C_{1}^j)^{k}(D_{1}^j)^{l-k}\big]m_{j}^{(k)}.\label{mom2}
\eeqn
\end{Proposition}
\begin{proof} We begin with the fact that for any Poisson random variable with parameter $w>0,$ we have $E\big(\text{Poi}(w)\big)_{k}= w^{k},$ and as a result of \eqref{stirling2}
\beqn
E\big([\text{Poi}(w)]^{n}\big)&=&\sum_{k=1}^{n}{n\brace k} E\Big[\big(\text{Poi}(w)\big)_{k}\Big]\non\\
&=& \sum_{k=1}^{n}{n\brace k} w^{k}\non
\eeqn
using the expansion with the Stirling number of type $2.$ 

Using the result of Theorem \ref{T1}, the above computation leads to 
\beqn
\int_{0}^{\infty}y^{n}\widetilde{\pi}_{1}(dy)&=&\sum_{j\in S}\pi_{j}E\big([\text{Poi}(W_{j})]^{n}\big)\non\\
&=& \sum_{j\in S}\pi_{j}\sum_{k=1}^{n}{n\brace k} E\big[W_{j}^{k}\big],\non\\
&=& \sum_{j\in S}\pi_{j}\sum_{k=1}^{n}{n\brace k}E\Big[\lambda(j)\int_0^{T_j}e^{-\mu(j)(T_{j}-s)}ds+e^{-\mu(j)T_j}V^{*}_j\Big]^{k}\non
\eeqn
given  $m^{(n)}_{j}=E\big(V^{*}_j\big)^{n}<\infty$ exists. Given the assumptions of Theorem \ref{T1} and since $\mu(\cdot)$ is non-negative and strictly positive for some $i\in S$, for all $m\in\mathbb{N}$ $Ee^{-m\int_{\tau^{j}_{0}}^{\tau_{1}^{j}}\mu(X_{s})ds}<1.$ Under second condition of \eqref{mom0}, the moments  of $V^{*}_j$ can be computed recursively using the representation for $(C_{1}^j, D_{1}^j)$ in \eqref{AB}. From $V^{*}_j\stackrel{d}{=}C_{1}^j V^{*}_j+D_{1}^j,$ for any integer $0\le l\le n$, taking $f(x)=x^{l}$ in both sides of the above recurrence equation yields
$$
(V^{*}_j)^{l}\stackrel{d}{=}(C_{1}^j)^l (V^{*}_j)^{l}+\sum_{k=0}^{l-1}{l \choose k}(C_{1}^j)^k (D_{1}^j)^{l-k}(V_j^{*})^{k}.
$$
Under conditions of \eqref{mom0}, $E\big[(C_{1}^j)^{l}\big]< 1$ and $E[(C_{1}^j)^k (D_{1}^j)^{l-k}]<\infty$ for $0\le k\le l\le n$, then recursively one has $E|V^{*}_j|^{n}<\infty.$ Independence between $V^{*}_j$ and $(C_{1}^j,D_{1}^j)$ gives the following recursive relation of the  moments
\beqn
m_{j}^{(l)}:=E[(V^{*}_j)^{l}]&=&\frac{1}{1-E(C_{1}^j)^{l}}\sum_{k=0}^{l-1}{l\choose k}E\big[(C_{1}^j)^{k}(D_{1}^j)^{l-k}\big]E[(V^{*}_{j})^{k}]\non\\
&=&\frac{1}{1-E(C_{1}^j)^{l}}\sum_{k=0}^{l-1}{l\choose k}E\big[(C_{1}^j)^{k}(D_{1}^j)^{l-k}\big]m_{j}^{(k)}\non.
\eeqn
which matches identically with \eqref{mom2}. The expectation in the RHS of \eqref{mom1} can be computed using the independence of $T_j, V_{j}^{*},$ along with the values of $\{m_{j}^{(k)}:1\le k \le n\}$. 
\hfill$\square$
\end{proof}

\section{Sampling scheme}\label{SS}Theorem \ref{T1} provides a way to characterize the limiting bi-variate measure (total counts along with the environment) $$\widetilde{\pi}:=\mathcal{L}\bigg(\sum_{j\in S}^{}\delta_{U_{*}}(\{j\})\text{Poi}(W_{j}),U_{*}\bigg)$$ under semi-Markovian switching environment and some conditions, but how can we use that to simulate observations from the measure, given that the environmental state $\{U_{*}=j\}$ is known? The goal of this section is to address how to sample from the conditional limiting measure
$$\widetilde\pi(\cdot\,|\, j):=\frac{\widetilde\pi(\cdot,j)}{\pi_{j}}=P\big[\text{Poi}(W_{j})\in\cdot\big].$$

The computation of such quantities is often useful when some environmental state is largely observable. For various statistical inferences, it is important to study the properties of the distribution of $\widetilde{\pi}.$ One way to do this is to generate a sample $j$ from the known distribution $\pi$ (follows from the dynamics $X$), and then generate from $\widetilde\pi(\cdot\,|\, j).$  Moreover, if we want to make statistical inferences about rare events, such as "given the environment is set at  $j\in S,$ what will be the propensity or mean of the total number of people in the system exceeding a certain given threshold?", we need to generate samples from $\widetilde\pi(\cdot\,|\, j),$ which is  $\mathcal{L}(\text{Poi}(W_{j}))$ as shown above. 

From Theorem \ref{T1}, the distributional structure of $\mathcal{L}(W_{j})$ is identical as \newline $\mathcal{L}\big(\lambda(j)\int_0^{T_j}e^{-\mu(j)(T_{j}-s)}ds+e^{-\mu(j)T_{j}}V^{*}_j\big),$ where $T_{j}\sim \pi_{j}^{*}$ is independent with $V_{j}^{*}.$ Hence to generate one sample from $\mathcal{L}(W_{j}),$ it is enough to generate one sample from $\mathcal{L}(V^{*}_j).$

We assume the conditions of Theorem \ref{T1}. From one trajectory of $X,$ we may chop the path $(0,\tau_{n}^{j}]$ in disjoint intervals $\cup_{i=1}^{n} (\tau_{i-1}^{j},\tau_{i}^{j}]$ and from that we may construct the i.i.d samples for $i=1,\ldots,n$

$$(C^{j}_{i},D^{j}_{i})=\Big(e^{-\int_{\tau^{j}_{i-1}}^{\tau^{j}_{i}}\mu(X_{s})ds},\int_{\tau^{j}_{i-1}}^{\tau^{j}_{i}}\lambda^{}(X_{s})e^{-\int_{s}^{\tau^{j}_{i}}\mu(X_{r})dr}ds\Big),$$ that are considered as the iid samples from joint bi-variate distribution \newline $\mathcal{L}\Big(e^{-\int_{\tau^{j}_{0}}^{\tau^{j}_{1}}\mu(X_{s})ds},\int_{\tau^{j}_{0}}^{\tau^{j}_{1}}\lambda^{}(X_{s})e^{-\int_{s}^{\tau^{j}_{1}}\mu(X_{r})dr}ds\Big).$ Now generate an observation  $T^{j}$ from the distribution $\pi_{j}^{*}$ (i.e, size-biased distribution of averaged sojourn time at state $j\in S$) which is independent of $\{(C^{j}_{i},D^{j}_{i}):i=1,\ldots,n\}.$  Let $V_{(j,0)}=0$ and define recursively for  $1\leq i\leq n$,
$$V_{(j,i)}=C_{i}^{j} V_{(j,i-1)}+D_{i}^j.$$
Considering $V_{(j,n)}$ as an approximate sample for $V_{j}^{*}$, we consider $V^{}_{j,n}:=e^{-\mu(j)T_{j}}V_{(j,n)}+\lambda(j)\int_0^{T_{j}}e^{-\mu(j)(T_{j}-s)}ds,$ as an approximate sample from $\mathcal{L}(W_{j})$. The following result gives a non-asymptotic estimate of the error made with this approximation, in terms of the Wasserstein-$1$ metric. This also gives a criteria for selecting $n,$ to obtain the desired accuracy.

\begin{Proposition}\label{P3}
Suppose assumptions of Theorem \ref{T1} hold, and added to that assume further 
\beqn
E_{\pi}\lambda(\cdot)<\infty.\label{finL}
\eeqn
 Then for any $n\ge 1,$ there exists some finite $A_{1}$ and $r>0$ (both $A_{1},r$ may depend on $j$) such that,
$$\mathcal{W}_{1}\Big(\mathcal{L}(V^{}_{j,n}), \mathcal{L}(W_{j})\Big)\leq  A_{1} e^{-rn}.$$
\end{Proposition}
\begin{remark}\label{Re}
The condition \eqref{finL} being not very restrictive in this context, implies that $E[D_{i}^{j}]<\infty,$ which makes both $V^{}_{j,n}, V_{j}^{*}$ having finite means, justifying the use of Wasserstein-$1$ norm. If we choose other norms, then conditions weaker than \eqref{finL} may be used. 
\end{remark}
\begin{proof} Observe that ``$E_{\pi}\mu(\cdot)>0"$ is equivalent to $``P\big[\int_{\tau_{0}^{j}}^{\tau_{1}^{j}}\mu(X_{s})ds>0\big]>0,"$ where in general $\int_{\tau_{0}^{j}}^{\tau_{1}^{j}}\mu(X_{s})ds\ge 0$ a.s. Hence for any $i\in\mathbb{N},$ $E[C_{i}^{j}]<1,$ and almost surely $e^{-\int_{s}^{\tau^{j}_{i}}\mu(X_{r})dr}\le 1,\s \forall  s\in[\tau_{0}^{j},\tau_{1}^{j}).$ As a consequence, 
$$E[D_{i}^{j}]=E\Big[\int_{\tau^{j}_{i-1}}^{\tau^{j}_{i}}\lambda^{}(X_{s})e^{-\int_{s}^{\tau^{j}_{i}}\mu(X_{r})dr}ds\Big]\le E\Big[\int_{\tau^{j}_{i-1}}^{\tau^{j}_{i}}\lambda^{}(X_{s})ds\Big]=\big(E_{\pi}\lambda(\cdot)\big)E|I_{j}|<\infty.$$
Using the expression $V_{(j,n)}=\sum_{i=1}^{n}\big(\prod_{k=i+1}^{n}C_{k}^{j}\big)D_{i}^{j},$ it follows that for any $n\in \mathbb{N},$
\begin{eqnarray}
E(V_{(j,n)}) =\sum_{k=1}^{n}E(D_{i}^{j})E(C_{i}^{j})^{k-1}\le \frac{E(D_{i}^{j})}{1 -E(C_{i}^{j})}\le\frac{\big(E_{\pi}\lambda(\cdot)\big)E|I_{j}|}{1 -E(C_{i}^{j})}<\infty,\label{ub1}
\end{eqnarray}
that is a consequence of the fact that $(C_{i}^{j},D_{i}^{j})$ are iid. This upper bound will also hold for $EV_{j}^{*}.$

Take $$A_{1}:=\frac{\big(E_{\pi}\lambda(\cdot)\big)E|I_{j}|}{1 -E(C_{i}^{j})},\s r:=-\log(E(C_{i}^{j})).$$

Let $T_{j}\sim \pi_{j}^{*},$ be independent of $\{(C_{i}^{j},D_{i}^{j}):i\ge 1\}.$ Observe that, $V_{(j,n)}=\sum_{i=1}^{n}\big(\prod_{k=i+1}^{n}C_{k}^{j}\big)D_{i}^{j}= \widetilde{P}_{n}^{}(C,D)$  where $\widetilde{P}_{n}^{}(C,D),P_{n}^{j}(C,D)$ were defined in \eqref{PP}.

As a consequence of $\{(C_{i}^{j},D_{i}^{j}):i\ge 1\}$ being an iid sequence, following holds
\beqn
\mathcal{L}\Big((C_{1}^{j},D_{1}^{j}),(C_{2}^{j},D_{2}^{j}),\cdots,(C_{n}^{j},D_{n}^{j})\Big)=\mathcal{L}\Big((C_{n}^{j},D_{n}^{j}),(C_{n-1}^{j},D_{n-1}^{j}),\cdots,(C_{1}^{j},D_{1}^{j})\Big).\label{lhrh}
\eeqn

Since $\widetilde{P}_{n}^{}(C,D)$ and $P_{n}^{j}(C,D)$ are outputs of identical functions of the random variables in the LHS and RHS of \eqref{lhrh}, it follows that \[\widetilde{P}_{n}^{}(C,D)\stackrel{d}{=}P_{n}^{}(C,D)\s \text{for each}\,\, n\ge 1.\]
Hence $V_{(j,n)}= \widetilde{P}_{n}^{}(C,D)\stackrel{d}{=} P_{n}^{}(C,D).$  
From the Proof of Claim $2$ in subsection \ref{lemPerp} we have $V_{j}^{*}\stackrel{d}{=}P_{\infty}^{}(C,D),$ that satisfies the unique solution to \eqref{x=ax+b} with $(C,D)$ in \eqref{AB}. Hence, 
\begin{align*}
\mathcal{W}_{1}\big(\mathcal{L}(V^{}_{j,n}), \mathcal{L}(W_{j})\big)&=\mathcal{W}_{1}\Big(\mathcal{L}(e^{-\mu(j)T_{j}}V_{(j,n)}+G_{j}^{\mu,\lambda}(T_{j})), \mathcal{L}(e^{-\mu(j)T_{j}}V^{*}_{j}+G_{j}^{\mu,\lambda}(T_{j}))\Big)\non\\
&=\mathcal{W}_{1}\Big(\mathcal{L}(e^{-\mu(j)T_{j}}P_{n}^{}(C,D)+G_{j}^{\mu,\lambda}(T_{j})), \mathcal{L}(e^{-\mu(j)T_{j}}P_{\infty}^{}(C,D)+G_{j}^{\mu,\lambda}(T_{j}))\Big)\non\\
&\le E(|e^{-\mu(j)T_{j}}P_{n}^{}(C,D)+G_{j}^{\mu,\lambda}(T_{j})-e^{-\mu(j)T_{j}}P_{\infty}^{}(C,D)-G_{j}^{\mu,\lambda}(T_{j})|)\\
&\le E(e^{-\mu(j)T_{j}})E(|P_{\infty}^{}(C,D) - P_{n}^{}(C,D)|)\\&\le E(|P_{\infty}^{}(C,D) - P_{n}^{}(C,D)|).
\end{align*}

Note that \[P_{\infty}^{}(C,D) - P_{n}^{}(C,D)= \Big(\prod_{i=1}^{n}C_{i}^{j}\Big)\bigg[\sum_{k=n+1}^{\infty}\Big(\prod_{i=n+1}^{k-1}C_{i}^{j}\Big)D^{j}_{k}\bigg].\]

The assertion follows by observing that the two quantities in the RHS are independent and the first quantity after taking expectation yields $e^{-rn},$ while the second one after expectation is upper-bounded by $A_{1}.$
\hfill$\square$
\end{proof}

\subsection{Exceedance event computation}
Suppose we are interested in knowing the probability $\widetilde{\pi}_{1}[c,\infty)$ for some threshold $c>0$ (which is large but not large enough to use tail asymptotic type approximations), then using the definition in \eqref{pi1} it follows that 
\[\widetilde{\pi}_{1}[c,\infty)=\sum_{j\in S}\pi_{j}P[\text{Poi}(W_{j})\ge c].\]
Now for a fixed $j\in S,$ if we can generate $K$ iid samples of $\{(T^{(i)}_{j},V^{(i)}_{(j,n)}): i=1,\ldots,K\}$ using the sampling method described for certain $n,$ then clearly $P[\text{Poi}(W_{j})\ge c]$ can be estimated by $$\frac{1}{K}\sum_{i=1}^{K}1_{\Big\{\text{Poi}\big(e^{-\mu(j)T^{(i)}_{j}}V^{(i)}_{(j,n)}+\lambda(j)\int_0^{T^{(i)}_{j}}e^{-\mu(j)(T^{(i)}_{j}-s)}ds\big)\ge c\Big\}}$$ for some large $K.$  Alternatively one may use $R$-programming language to calculate the Poisson probabilities (using a command like \textit{ppois$(\cdot)$}) with rate that is generated from $\mathcal{L}(W_{j}),$ and then take the Monte-Carlo average of these probabilities (instead of above identity function) to get an estimate of $\widetilde{\pi}_{1}[c,\infty)$ for certain $c>0.$ This gives an efficient  estimator of the exceedance probability $\widetilde{\pi}_{1}[c,\infty),$ that would be difficult to find otherwise.

Consider the following two examples of an infinite server queue when the arrival and service rates are both modulated by the stochastic environment $X$. 

\subsection{Example $1$:}
Suppose $X:=(X_{1},X_{2})$ is a $S:=\{(k,10-k): k=0,\ldots,10\}$-valued semi-Markov processes. Clearly $X_{1}+X_{2}=10.$ We describe the $P$ matrix governing the transition kernel of a time-homogeneous discrete time state-wise Markov chain for $X_{1}$. Let $P=\{P_{ij}:i=0,\ldots,10\}$ be such that $P_{ij}=P[U_{k+1}=j\mid U_{k}=i],$ and $P_{ii}=0$ for all $i=0,\ldots,10.$ We specify the transition rates as follows
$$P_{i,i-1}=\frac{4i}{50-i},\,\,\,\, P_{i,i+1}=\frac{5(10 -i)}{50-i},\s  i=1,\ldots,9$$
 $P_{0,1}= P_{10,9}= 1,$ and $P_{ii}=0$ for all $i=0,\ldots,10.$ We specify the set $\mathcal{G}$ of probability distributions of sojourn times of $X_{1},$ as $\mathcal{G}_{\text{Exp}}=\{F_{ij}= \text{Exp}(50 - i), i=0,\ldots,10\}.$

\begin{figure}[h]
\centering
\includegraphics[width=16cm, height=6cm]{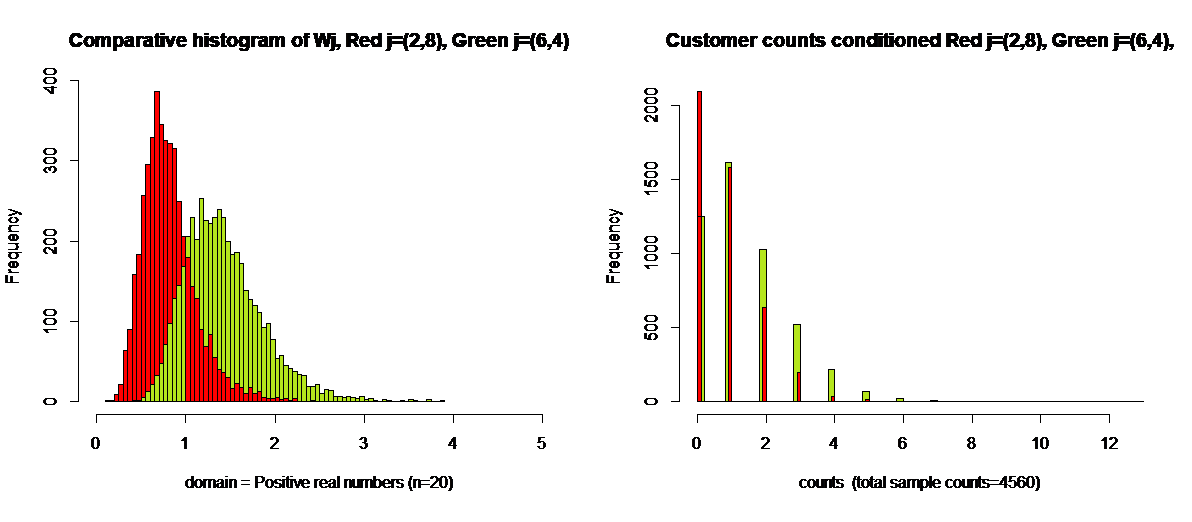}
\caption{In \textbf{Example $1$} we generated samples from mixture measure $\mathcal{L}(W_{\mb{j}})$ applying the sampling scheme for $n=20,$ for $\mb{j}=(2,8)$ and $\mb{j}=(6,4)$ and the histograms of sampled Poisson counts with the mixture $W_{\mb{j}}$ are also presented}\label{ex11}
\end{figure}

Clearly, under the given setup, $X$ is a birth-death process. Using its properties, the stationary distribution of $X$ can be expressed as $$\mu_{i}=\lim_{t\to\infty}P[X_{1}(t)=i, X_{2}(t)=10-i]={10 \choose i}\Big(\frac{5}{9}\Big)^{i}\Big(\frac{4}{9}\Big)^{10-i}.$$

Now we consider infinite server queue with arrival rate $\lambda(X_{1},X_{2})=X_{1}$ and the service rate $\mu(X_{1},X_{2})=X_{2}.$ From Theorem \ref{T1}, the conditional limiting distribution becomes (for any environment fixed at $(X_{1},X_{2})=\mb{j}:=(j,10-j)$)
\begin{equation*}
\widetilde\pi(z\,|\, \mb{j}):= \int_{\R_{+}}P\Big[\textup{Pois}(y)=z\Big]\mu^{(W)}_{\mb{j}}(dy)=\int_{\R_{+}}\Big(\frac{y^z}{z!}e^{-y}\Big)\mu^{(W)}_{\mb{j}}(dy) 
\end{equation*}
where $\mu^{(W)}_{\mb{j}}=\mathcal{L}(W_{\mb{j}}).$   Furthermore,  $\mathcal{L}(W_{\mb{j}})=\mathcal{L}(e^{-\mu(\mb{j})T_{\mb{j}}}V^{*}_{\mb{j}}+G^{\mu,\lambda}_{\mb{j}}(T_{\mb{j}}) ))$, where $\lambda(\mb{j})=j, \mu(\mb{j})=10 -j,$ and $T_{\mb{j}}$ is an exponential random variable with rate $(50-j)$ that is independent of $V^{*}_{\mb{j}}$, such that
\begin{eqnarray}
V^{*}_{\mb{j}}&\stackrel{d}{=} &e^{-\int_{\tau^{j}_{0}}^{\tau^{j}_{1}}\mu(X_{s})ds}V^{*}_{\mb{j}}+ \int_{\tau^{j}_{0}}^{\tau^{j}_{1}}\lambda^{}(X_{s})e^{-\int_{s}^{\tau^{j}_{1}}\mu(X_{r})dr}ds,\s\s V^{*}_{\mb{j}}\perp \big(X_{s}: s\in [\tau_{0}^{j},\tau_{1}^{j})\big) \non
\end{eqnarray}
Using the simulation scheme taking $n=20,$ histograms of $\widetilde\pi(z\,|\, \mb{j})$ are given below for $\mb{j}=(2,8)$ and $\mb{j}=(6,4)$ are following. It looks like $\mathcal{L}(W_{\mb{j}})$ for $\mb{j}=(2,8)$ has more peak with peak location somewhere at less than $1,$ where for $\mb{j}=(6,4)$ $\mathcal{L}(W_{\mb{j}})$ is more skewed towards right and has a smaller peak between $(1,2).$

\subsection{Example $2$:}
In the previous example, the latent process had a finite state-space. Contrary to that, we consider another example of an underlying latent process $X$ that is semi-Markovian with state dynamics $U$ being an $S=\{0,1,2,\ldots\}$-valued discrete-time Markov chain with a transition kernel $$P_{i,i+1}=\frac{\lambda}{i+1},\s P_{i,0}=1-P_{i,i+1},\s \text{for any } i\in S,$$
where $\lambda$ is an arbitrary parameter such that $0<\lambda<1$. It can be shown that the stationary distribution $\mu$ for $U$ will be Poisson$(\lambda).$ For both following cases, we consider $\lambda=1.$

Now we consider two cases of different sojourn time dynamics: Exponential and Pareto. For the first case we assume  $\mathcal{G}_{\text{Exp}}=\{F_{ij}= \text{Exp}(3i+1), i\in S\},$ and for the second case $\mathcal{G}_{\text{Pareto}}=\big\{F_{ij}= \mathcal{L}\big(\text{Pareto}(1,\alpha(i))-1\big), i\in S\big\},$ where $\alpha(i)>2,\,\,\forall i\in S$ and $\text{Pareto}(x_{m},\alpha)$ is a Pareto random variable of type $1$, with scale parameter $x_{m},$ and shape parameter $\alpha$. In the second example, the scale parameter is $1$ for all sojourn times, but then we change the location by $1,$ which means that all the sojourn times are supported on the $[0,\infty).$ The mean of the law of $i$-th sojourn time $\mathcal{L}\big(\text{Pareto}(1,a(i))-1\big)$  is $\frac{1}{\alpha(i)-1}.$

In the first \textbf{Exponential} case, we consider $\lambda(X):=1+2X,\,\, \mu(X):=X$ as the $X$ dependent arrival and service rate. Observe that $\mathcal{L}(W_{j})=\mathcal{L}(e^{-\mu(j)T_{j}}V^{*}_{j}+G^{\mu,\lambda}_{j}(T_{j}) ))$, where $\lambda(j)=1+2j, \mu(j)=j,$ and $T_{j}$ is an Exponential random variable with rate $(3j+1)$ that is independent of $V^{*}_{j}.$ Using the sampling scheme for $j=0,\,\,\&\,\, j=3$ a comparative diagram of the histogram density of the mixing measure $\mathcal{L}(W_{j})$ follows.

\begin{figure}[h]
\centering
\includegraphics[width=16cm, height=6cm]{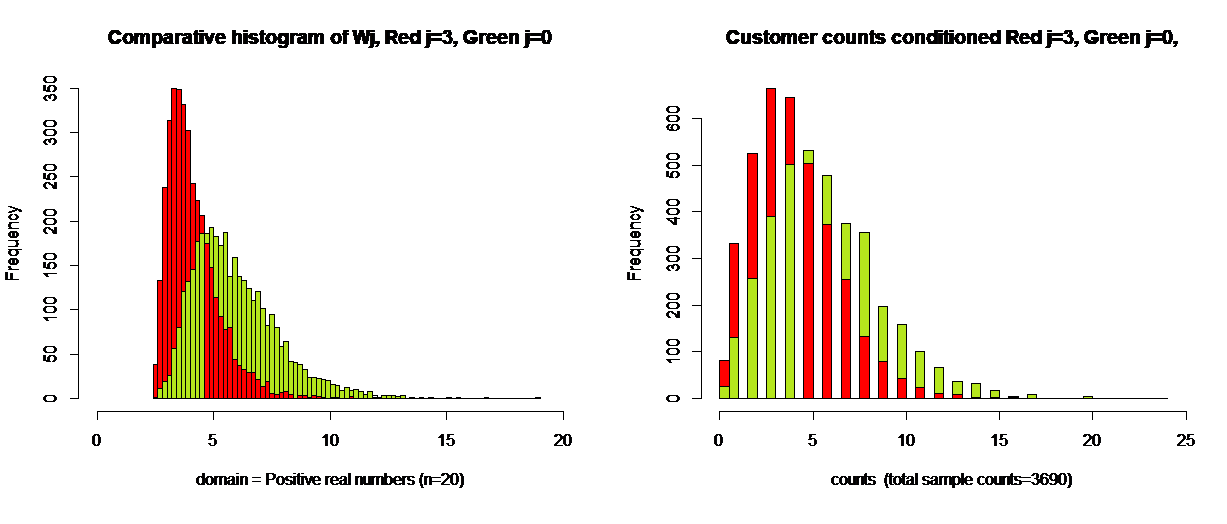}
\caption{In \textbf{Example $2$} for $\mathcal{G}_{\text{Exp}}$ we generated samples from mixture measure $\mathcal{L}(W_j)$ applying the sampling scheme (in Section \ref{SS}) for $n=20,$ for $j=3$ and $j=0$ and the histograms of sampled Poisson counts with the mixture $W_{j}$ are also presented}\label{ex21}
\end{figure}

In the second \textbf{Pareto} case (Figure \ref{ex22P}), we consider $\lambda$ and $\mu$ as before and compare the histogram density of the mixing measure $W_j$ for $j=3$ and $j=0$, with $\alpha(i) = \alpha = 2.2$ for all $i \in S$. Observe that $\mathcal{L}(W_j) = \mathcal{L}(e^{-\mu(j)T_j} V^{*}_j + G^{\mu,\lambda}_j(T_j))$, where $\lambda(j) = 1 + 2j$, $\mu(j) = j$, and $T_j$ is a Pareto random variable with rate $\alpha - 1 = 1.2$, independent of $V^{*}_j$. Hence it follows that
\[W_3 = e^{-3T_3} V^{*}_{3} +\frac{7}{3}[1-e^{-3T_{3}}], \s W_0 = V^{*}_0 + T_0 \]
where $T_0 \perp V^{*}_0,\,\, T_{3}\perp V^{*}_{3},\,\, T_0,T_3 \sim \text{Pareto}$ of type $1$ with the shape parameter $1.2$. For $\alpha = 2.2$ and $j = 0, 3$, the comparative histograms of the mixing measure $\mathcal{L}(W_j)$ are presented in the Figure \ref{ex22P}.

\begin{figure}[h]
\centering
\includegraphics[width=16cm, height=6cm]{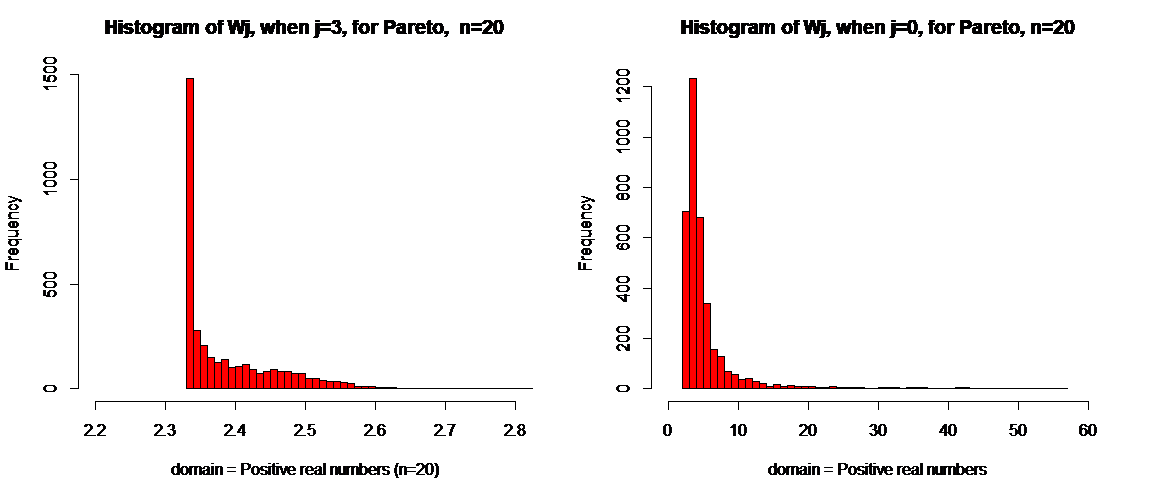}
\caption{In \textbf{Example 2} for $\mathcal{G}_{\text{Pareto}}$, samples were generated from the mixing measure $\mathcal{L}(W_j)$ using the sampling scheme for $n=20$, for $j=3$ and $j=0$. The histograms of the sampled Poisson counts with the mixture $W_j$ are also shown.}\label{ex22P}
\end{figure}

It appears that the mixing measure corresponding to $X=3$ (a rarer event than $X=0$) is much more peaked, with a narrower domain, similar to the case of $\mathcal{G}_{\text{Exp}}$. However, the nature of the difference is notable due to the Pareto sojourn times. The distribution $\mathcal{L}(W_j)$ for $j=0$ is highly right-skewed, with a heavy tail extending towards $+\infty$. Approximately 1\% of the observations, exceeding 60, were discarded in the histogram, with the maximum value of the samples reaching around $576.9$ whereas for $\mathcal{L}(W_{3})$ it is $2.915$. This is because, for $j=0$, the effect of $T_j$ on $W_j$ is linear without any exponential factor (since $\mu(0) = 0$). This is due to the additive effect of the heavy-tailed $\pi_{0}^{*}$,i.e, $W_0 = T_0 + V^{*}_0$, where $T_0$ follows a Pareto distribution with $\alpha = 1.2$. Therefore, $\mathcal{L}(W_0)$ will exhibit heavy-tailed behavior in the right tail due to $T_0$, while its behavior near the mean will be governed by $V^{*}_0$. This phenomenon would not hold for other $W_j$ for $j\neq 0,$ as the additive effect is bounded by $\frac{\lambda_{j}}{\mu_{j}}.$

In all three above cases (Example $1$ \& $2$), the distribution $\mathcal{L}(W_{j})$ for a state $j$ that is closer to the mean of $X$ tends to have more spread and less peaked. In contrast, for a rarer state $j$ that lies further from the mean, $\mathcal{L}(W_{j})$ exhibits a sharper peak and a narrower spread.

\section{Sharpness of the Assumption \ref{A0}}\label{example} We give an example where Assumption \ref{A0} does not hold and ergodicity of $Y$ will be compromised under certain parameter values. Suppose people are coming to a facility consisting of infinite workstations, and two identical servers parallelly supply software-based support to each of these infinite workstations. Whenever a person enters she is automatically assigned to a workstation. At any time only one server can support these workstations. 

Let $X_{1}$ be a $\{0,1\}$-valued continuous-time Markov chain with rate matrix $Q_{1}=\begin{pmatrix} -\lambda_{0} & \lambda_{0}\\ \lambda_{1} & -\lambda_{1}
\end{pmatrix},$ indicating the activity status of the two servers. $\{X_{1}(t)=0\}$ implies at time $t>0,$ the first server is active while the second server is dormant/resting, and $\{X_{1}(t)=1\}$ implies the reverse.

$X_{2}$ is a $\{0,\ldots, k\}$ valued stochastic process, denoting the allocated bandwidth level of the active server, where level $0$ means ``no bandwidth allocated" and level $k$ means ``full bandwidth allocated." If the active server is allocated $x$ level of bandwidth, then the other server is allocated $(k-x)$ bandwidth level, i.e, the total bandwidth level of both servers will add up to $k$. Whenever one person's job is finished and she leaves the system, instantaneously the allocated bandwidth level of the active server drops by one which is re-allocated to the in-active server. The service rate of all workstations is identically assumed to be proportional to the allocated bandwidth level of the active server (but the multiplicative constants may differ depending on the server status $X_{1}$), i.e if we denote the background switching environment $X=(X_{1}, X_{2})$ then the service rate can be expressed as
\beqn
\mu(x):= \big[(1-x_{1}) q_{1}+ x_{1}q_{2}\big]x_{2}1_{\{x_{2}>0\}}, \s\s x=(x_{1},x_{2})\in \{0,1\}\times\{0,\ldots,k\}\label{serv}
\eeqn
for some constants $q_{1},q_{2}\in\R_{+}.$  Jointly $(X_{1},X_{2},Y)$ is a $\{0,1\}\times\{0,\ldots,k\}\times \mathbb{N}$-valued continuous-time Markov chain with rate function $Q,$ which is specified by the following rates. 

For any $(x_{1},x,y)\in \{0,1\}\times\{0,\ldots,k\}\times \mathbb{N}$
\beqn
&(0,x,y)\to (1,k-x,y)& \s\text{ with rate,}\s \lambda_{0},\,\,\non\\
&(1,x,y)\to (0,k-x,y)& \s\text{ with rate,}\s \lambda_{1},\,\,\non\\
&(x_{1},x_{},y)\to(x_{1},x_{},y+1)& \s\text{ with rate,}\s \lambda, \non\\
&(0,x,y)\to (0,x-1,y-1)&\s\text{ with rate,}\s q_{1}xy,\non\\
&(1,k-x,y)\to (1,k-x-1,y-1)&\s\text{ with rate,}\s q_{2}(k-x)y.\non
\eeqn

Marginally the dynamics of $X_{1}$ will be governed by a continuous-time Markov chain with rate $Q_{1}.$ Whenever the value of $X_{1}$ changes, $X_{2}$ will change its value from $x_{2}\to k-x_{2}$ for some $x_{2}\in\{0,\ldots,k\}.$ The dynamics of $Y$ for any $t\ge 0,$ marginally can be described as
\beqn
Y_{t} &\stackrel{d}{=}&  Y_{0}  + N_{1}(\lambda t)  - N_{2}\Big(\int_{0}^{t}Y_{s-}\big[(1-X_{1}(s)) q_{1}+ X_{1}(s)q_{2})\big]X_{2}(s) 1_{\{X_{2}(s)>0\}}ds\Big)\label{YX2}
\eeqn
where $(N_{i}: i=1,2)$ are independent unit rate Poisson processes. Marginally $(Y_{s}:s\ge 0)$ behave identically with \eqref{inftysq} with $\lambda(\cdot)=\lambda$ (a constant) and $\mu(\cdot)$ as defined in \eqref{serv}. The marginal evolution of $(X_{2}(s):s\ge 0)$ depends on $X_{1},$ as well as on the feedback from $Y$, as it reduces by one when one customer leaves the system after her job is done. Hence Assumption \ref{A0} is invalid here for $X=(X_{1},X_{2})$ and $Y$. The following result shows that for a certain choice $k,\lambda_{0},\lambda_{1}$ if the arrival intensity rate $\lambda$ is more than a threshold, then $Y$ will be transient. 

\begin{Proposition}\label{P1}
If $k<\frac{\lambda}{2}\Big(\frac{1}{\lambda_{1}}+\frac{1}{\lambda_{0}}\Big),$ then the process $Y:=(Y_{s}:s\ge 0)$ is transient,
irrespective of any other parameters.
\end{Proposition}
\begin{proof}
The marginal dynamics of $X_{1}$ will be a stationary Markov chain with regenerative renewal property. We track what happens for the dynamics of $Y$ in one regenerating cycle of $X_{1}$ in $\{1\to 0, 0\to 1\}$ or in $(\tau_{0}^{1},\tau_{1}^{1}]$. Observe that during the sojourn at $\{X_{1}=1\}$ in $(\tau_{0}^{1},\tau_{1}^{1}],$  at most $k$ people can leave the system, since the bandwidth can drop at most $k$- level and once the bandwidth is $0,$ then all service stops until $X_{1}$ switches its value from $1\to 0$. Similarly while $X_{1}=0$ in that regenerating interval, at most $k$ people can leave the system.

Moreover the expected total people arriving in the interval $(\tau_{0}^{1},\tau_{1}^{1}]$ is more than $2k$, since
\beqn
E\big[Y_{\tau_{1}^{1}}-Y_{\tau_{0}^{1}}\big]&\ge& E\Big[\int_{\tau_{0}^{1}}^{\tau_{1}^{1}}N_{1}(\lambda dt)\Big] - 2k\non\\
&=&\lambda\Big(\frac{1}{\lambda_{1}}+\frac{1}{\lambda_{0}}\Big)-2k>0.\non
\eeqn
It follows that using above arguments $E\big[Y_{\tau_{i}^{1}}-Y_{\tau_{i-1}^{1}}\big]\ge \lambda\Big(\frac{1}{\lambda_{1}}+\frac{1}{\lambda_{0}}\Big)-2k>0$ holds for any $i\ge 1.$ Hence for every $n\ge 1$, $$Y_{\tau_{n}^{1}}-Y_{\tau_{0}^{1}}=\sum_{i=1}^{n}(Y_{\tau_{i}^{1}}-Y_{\tau_{i-1}^{1}})$$ will be the partial sum of increments, each having a strictly positive mean. By renewal law of large number, almost surely following holds
$$\liminf_{t\to\infty}\frac{Y_{t}}{t}>\lambda\Big(\frac{1}{\lambda_{1}}+\frac{1}{\lambda_{0}}\Big)-2k>0,$$
proving that $Y$ is transient. 
\hfill$\square$
\end{proof}
\section{Conclusion and future directions} \label{Conclu}

A framework for analyzing the steady-state behavior of an infinite-server queue is presented, where both the arrival and service rates are modulated by a stochastic environment that takes values in a countable (possibly infinite) set $S$ and exhibits regenerative-renewal property. Although we focus on a particular type of semi-Markov process (satisfying Assumption \ref{As0}) as an example, this framework can be extended to any other regenerative process. As $t \to \infty$, the system converges weakly to a mixture of Poisson-type distribution, where the mixing measure is explicitly expressed as an affine transformation of a stochastic recurrence equation (SRE) of the form $X \stackrel{d}{=} CX + D$, with $X$ being independent of $(C, D)$. Additionally, an approximate sampling scheme is proposed from that mixing measure, leveraging the explicit steady-state representation, and providing some convergence diagnostics to evaluate the accuracy of the approximation. We conclude by outlining some open questions and directions for future research.

\begin{enumerate}[(a)]
\item Can this framework be extended to more general regime processes while preserving the regenerative renewal structure? Partial insights from prior work on infinite-memory processes (e.g., \cite{berbee1987chains}, \cite{graham2021regenerative}) indicate that, under some restrictive conditions, the regenerative-renewal property of the underlying environment may still hold in cases of linear Hawkes process or other infinite-memory chains. However, applying such processes within our context requires a distributional estimate of the asymptotic residual time,
\[
\mathcal{L}\Big(t-\tau^{j}_{g_{t}^{j}} \mid X_{t} = j \Big),
\]
as well as developing a corresponding version of Lemma \ref{lem01} for such processes. These derivations may not be as explicit as it is in the semi-Markov setting discussed here, but they will be addressed in future.

\item As mentioned in Remark \ref{R0}, the renewal regenerative structure of $X$ will be absent if Assumption \ref{As0}(c) does not hold. Can we have any alternative methodology to circumvent this or at least have some closed-form limit results in such non-regular cases?
\item Is it possible to model feedback from $Y$ to $X$ in a queueing framework, where specific quantitative results on the long-term behavior can be determined? Alternatively, how can the transient behavior of the model described in Section \ref{example} be analyzed in a general formulation, perhaps using methods like diffusion approximation or some other techniques?
\end{enumerate}

\section{Proof of Theorem \ref{T1}}\label{proof}
\begin{proof}
Define the stochastic process $(\Phi,I^{(\lambda,\mu)}):=\Big((\Phi_{t}^{(\mu)},I_{t}^{(\lambda,\mu)}):t\ge 0\Big)$ such that
\beqn
\Phi_{t}^{(\mu)}=e^{-\int_{0}^{t}\mu(X_{r})dr},\s I_{t}^{(\lambda,\mu)}=\int_{0}^{t}\lambda^{}(X_{s})e^{-\int_{s}^{t}\mu(X_{r})dr}ds.\label{phiI}
\eeqn

We proceed by showing the following observation, which is a weak representation of $Y_{t}$ conditioned on the path of $X$ up to time $t$ :
\begin{equation}
\mathcal{L}(Y_{t}|\mathcal{F}^{X}_{t}) = \text{Bin}(Y_{0},\Phi_{t}^{(\mu)})\otimes \text{Poi}(I_{t}^{(\lambda,\mu)})\label{infin}
\end{equation}
using the notation \eqref{phiI} where $\otimes$ denotes the convolution operator.

Consider a person who arrives in the system at a time instant $u\in[0,t]$. The probability that this person will still be in the system at time $t$ can be expressed as:
\begin{equation}
P\left[N_{2,i}\left(\int_{u}^{t}\mu(X_{s})ds\right)=0\right] = e^{-\int_{u}^{t}\mu(X_{s})ds}.\label{MMinfty1}
\end{equation}

Next, we divide the total count of people $Y_{t}$ into two groups: $Y_{1}(t)$ and $Y_{2}(t)$. $Y_{1}(t)$ represents the number of people who are in the system at time $t$ out of the initial count $Y_{0}$ present in the system. $Y_{2}(t)$ represents the count of people who arrive in the system in the interval $(0,t)$ and still remain in the system at time $t$.

Conditioned on $\mathcal{F}^{X}_{t}$, both $Y_{1}$ and $Y_{2}$ are independent as $Y_{1}$ can be determined from the initial count $Y_{0}$, while $Y_{2}$ is determined by the independent arrival process $N_{1}(\cdot)$. Conditioned on $\mathcal{F}^{X}_{t}$, $Y_{1}(t)$ is interpreted as the number of people out of the initial count $Y_{0}$ who will remain in the system at time $t$. Thus, $Y_{1}(t)$ follows a binomial distribution:
\begin{equation}
\mathcal{L}(Y_{1}(t)| \mathcal{F}^{X}_{t}) = \text{Bin}(Y_{0},e^{-\int_{0}^{t}\mu(X_{s})ds})=\text{Bin}(Y_{0},\Phi_{t}^{(\mu)}).\label{Z1dist}
\end{equation}

Now focusing on $Y_{2}(t),$ we will show that it follows a Poisson distribution:
\begin{equation}
\mathcal{L}(Y_{2}(t)| \mathcal{F}^{X}_{t}) = \text{Poi}\left(\int_{0}^{t}\lambda(X_{s})e^{-\int_{s}^{t}\mu(X_{r})dr}ds\right).\label{infin2}
\end{equation}

Consider the total count of people arriving in the system in the interval $(0,t)$, denoted by $A_{t}$. Conditioned on $\mathcal{F}^{X}_{t}$, $A_{t}$ follows a Poisson distribution with mean $\int_{0}^{t}\lambda(X_{s})ds$. Arrival instances conditioned on $\mathcal{F}^{X}_{t}$ and $A_{t}$ are distributed as iid realizations from the density function $\frac{\lambda(X_{\cdot})}{\int_{0}^{t}\lambda(X_{u})du}$.

For any person arriving in the interval $(0,t)$, the conditional probability of being in the system at time $t$ is given by:
\[
\widetilde{P}_{t} = \frac{\int_{0}^{t}\lambda(X_{u})e^{-\int_{u}^{t}\mu(X_{s})ds}du}{\int_{0}^{t}\lambda(X_{u})du},
\]
where the exponential term represents the probability of a person staying in the system at time $t$ given that the person arrived at time $u$ (using \eqref{MMinfty1}).

By thinning of Poisson processes, we can multiply $\widetilde{P}_{t}$ with $\int_{0}^{t}\lambda(X_{s})ds$ to obtain \eqref{infin2}.

Finally, we note that the independence of people present at time $0$ and those who arrived afterward, given $\mathcal{F}^{X}_{t}$, implies the validity of \eqref{infin}.

Now we finish part (a) of the proof. Observe that as a consequence of \eqref{infin} unconditionally for any event $A\in\mathcal{B}(\R),$ $$P[Y_{t}\in A]=\int_{(0,1]\times \R_{\ge 0}}P\big[\text{Bin}(Y_{0},u_{1})+ \text{Poi}(u_{2})\in A\big]P_{(\Phi_{t}^{(\mu)},I_{t}^{(\lambda,\mu)})}(du_{1},du_{2}),$$
where $P_{(\Phi_{t}^{(\mu)},I_{t}^{(\lambda,\mu)})}(\cdot,\cdot)$ is a joint density of $(\Phi_{t}^{(\mu)},I_{t}^{(\lambda,\mu)})$.
Now we state the main Lemma of this article that will form the basis of the limiting measure.
\begin{lemma}\label{L1}
Under assumptions of Theorem \ref{T1}, the joint limiting measure of $(\Phi_{t}^{(\mu)},I_{t}^{(\lambda,\mu)},X_{t})$ can be given as 
$$(\Phi_{t}^{(\mu)},I_{t}^{(\lambda,\mu)},X_{t})\stackrel{d}{\to}\Big(0,\sum_{j\in S}\delta_{U_{*}}(\{j\})W_{j},U_{*}\Big)\s \text{as} \s t\to\infty,$$
where $U_{*}\sim \pi,$ and for each $j\in S,$ $W_{j}$ is defined in \eqref{mixture}.
\end{lemma}
(The proof of Lemma \ref{L1} is given in the next section)

Clearly the Binomial component in \eqref{infin} $$ \text{Bin}(Y_{0},\Phi_{t}^{(\mu)})\stackrel{P}{\to}0,$$ under unconditional $t\to\infty$ limit as $Y_{0}$ is fixed. Hence by Slutsky's theorem 
\begin{eqnarray}
\lim_{t\to\infty}\mathcal{L}(Y_{t},X_{t})&\stackrel{w}{=}&\lim_{t\to\infty}\mathcal{L}\big(\text{Poi}\big(I_{t}^{(\lambda,\mu)}\big), X_{t}\big)\non\\
 &\stackrel{w}{\to}&\Big(\sum_{j\in S}\delta_{U_{*}}(\{j\})\text{Poi}(W_{j}), U_{*}\Big)\s \text{as} \s t\to\infty\non
\end{eqnarray}
proving the assertion of the theorem. 
\hfill$\square$
\end{proof}

\section{Set-up for the proof of Lemma \ref{L1}}We begin with some notations $\sum_{i=j}^k a_i=0$ and $\prod_{i=j}^k a_i=1$ if $j>k$ for any $a_i$. Let $e^{-\int_{s}^{t}\mu(X_{r})dr}$ be denoted by $\Phi(s,t)$ and by this notation $\Phi_{t}=\Phi(0,t).$ Given two functions $\lambda(\cdot),\mu(\cdot): S\to \R$, we define some random variables $\big\{\big(C_{i}^{j},D_{i}^{j}\big):i\ge 1\big\}, \big(C_{0}^{j},D_{0}^{j}\big)$ as
\beqn
(C_{i}^{j},D_{i}^{j})&:=&\bigg(e^{-\int_{\tau_{i-1}^{j}}^{\tau_{i}^{j}}\mu(X_{s})ds},\int_{\tau_{i-1}^{j}}^{\tau_{i}^{j}}\lambda(X_{s})e^{-\int_{s}^{\tau_{i}^{j}}\mu(X_{r})dr}ds\bigg)\,\,\,\,\, \forall i\ge 1,\s\text{and}\label{LKintro1}\\
(C_{0}^{j},D_{0}^{j})&:=&\bigg(e^{-\int_{0}^{\tau_{0}^{j}}\mu(X_{s})ds},\int_{0}^{\tau_{0}^{j}}\lambda(X_{s})e^{-\int_{s}^{\tau_{0}^{j}}\mu(X_{r})dr}ds\bigg).\label{LKintro2}
\eeqn
Note that if $X_{0}=j$, then $\tau_{0}^{j}=0$ and $(C^{j}_{0},D^{j}_{0})=(1,0)$. $(I^{j}_{k})_{k\ge 1}$ represents the sequence of renewal cycles as an i.i.d. sequence and therefore also $(C^{j}_{i},D^{j}_{i})_{i=1}^{\infty}$ is an i.i.d. sequence under Assumption \ref{As0}. However, for fixed $t>0,$ $(C^{j}_{i},D^{j}_{i})_{i=1}^{g^{j}_{t}}$ is not an i.i.d. sequence since $g^{j}_{t}$, defined in \eqref{g_{t}}, is a renewal time which depends on the sum of all renewal cycle lengths before time $t$.

Given an arbitrary sequence of iid random variables $\{(C^{j}_{i},D^{j}_{i}):i\ge 1\}$ having common law $\mathcal{L}(C_{1}^{j},D_{1}^{j})$ (for simplicity we denote $\mathcal{L}(C,D)$), define the random variables $\big\{\big(P^{}_{n}(C,D),\widetilde{P}^{}_{n}(C,D)\big):n\ge 1\big\}$ as following 
\beqn
P^{}_{n}(C,D):=\sum_{k=1}^{n}\bigg(\prod_{i=1}^{k-1}C_{i}^{j}\bigg)D_{k}^{j},\s\s\widetilde{P}^{}_{n}(C,D):=\sum_{k=1}^{n}\bigg(\prod_{i=k+1}^{n}C_{i}^{j}\bigg)D_{k}^{j}.\label{PP}
\eeqn
Further note that in general $\{\widetilde{P}^{}_{n}(C,D):n\ge 1\}$ is Markovian and have a random coefficient autoregressive process type evolution $$\widetilde{P}^{}_{n+1}(C,D)=C_{n+1}^{j}\widetilde{P}^{}_{n}(C,D)+D_{n+1}^{j},\s\text{ with }\widetilde{P}^{}_{n}(C,D)\ci \big(C_{n+1}^{j},D_{n+1}^{j}\big),$$ while the sequence $\{P^{}_{n}(C,D):n\ge 1\}$ is not Markovian but admits the following representation $$P^{}_{n+1}(C,D)=P^{}_{n}(C,D)+\Big(\prod_{i=1}^{n}C_{i}^{j}\Big)D_{n+1}^{j}.$$  Observe that $\{(C_{i}^{j},D_{i}^{j}):i\ge 1\}$ are iid under Assumption \ref{As0}. For notational simplicity, by $P_{g_{t}^{j}}^{j},\widetilde{P}_{g_{t}^{j}}^{j}$ we denote $P_{g_{t}^{j}}^{}(C_{i}^{j},D_{i}^{j}),\widetilde{P}_{g_{t}^{j}}^{}(C_{i}^{j},D_{i}^{j})$ respectively in \eqref{PP}. Define $(L_{t})_{t\ge 0}$ as a process and $(Q^{(1)}_{t},Q^{(2)}_{t})_{t\ge 0}$ as an $(\mathcal{F}_{t}^{X})_{t\ge 0}$-adapted process defined respectively as
\beqn
L_{t}:=\big(L_{t}^{(1)}, L_{t}^{(2)}\big):=\bigg(\bigg(\prod_{k=1}^{g_{t}^{j}}C^{j}_{k}\bigg)^{} C^{j}_{0},\s\bigg(\prod_{k=1}^{g_{t}^{j}}C^{j}_{k}\bigg)^{} D^{j}_{0}+P^{j}_{g_{t}^{j}}\bigg)\label{Rt}
\eeqn
and 
\beqn
 (Q^{(1)}_{t},Q_{t}^{(2)})&:=&\bigg(e^{ -\mu(j)(t-\tau_{g_{t}^{j}}^{j})}L_{t}^{(1)},\s G_{j}^{\mu,\lambda^{}}\big(t-\tau_{g_{t}^{j}}^{j}\big)+ e^{-\mu(j)(t-\tau_{g_{t}^{j}}^{j})}L_{t}^{(2)}\bigg). \label{joint}
\eeqn
 From aforementioned definitions, both $L_{t}$ and $(Q^{(1)}_{t},Q_{t}^{(2)})$ depend on $j\in S$ but for notational simplicity $j$ is omitted. 

By the notation $\mathcal{F}^{X}_{\tau_{g^{j}_{t}}^{j}}$, we define the sigma field generated by the class of sets $\{A: A\cap \{g_{t}^{j}=n\} \in \mathcal{F}^{X}_{\tau_{n}^{j}}\}$. Crucial elements of proof of Lemma \ref{L1} rely on the following lemma concerning the asymptotic conditional independence of $t - \tau_{g_{t}^{j}}^{j}$ and a $\Big\{\mathcal{F}^{X}_{\tau_{g^{j}_{t}}^{j}}\Big\}_{t \ge 0}$-measurable $\mathbb{R}^{}$-valued process $(\mathcal{J}_{t}: t \ge 0)$, conditioned on $\{X_{t} = j\}$. We present Lemma \ref{lem01} separately, as it is useful for eliminating the conditioning on $X_{t} = j$ in the limit as $t \to \infty$, particularly when the main event concerns a functional of $X$, which exhibits a marginal weak limit with certain properties (see \eqref{key}).

\begin{lemma}\label{lem01} 
Suppose $X$ is a semi-Markov process satisfying Assumption \ref{As0}. For $j \in S,$ suppose $(\mathcal{J}_{t}: t \ge 0)$ is a $\Big\{\mathcal{F}^{X}_{\tau_{g^{j}_{t}}^{j}}\Big\}_{t \ge 0}$-measurable, $\mathbb{R}^{}$-valued process such that there exists a random variable $\mathcal{J}_{\infty}$ and an increasing deterministic function $\epsilon(\cdot): \mathbb{R}_{\ge 0} \to \mathbb{R}_{\ge 0}$, satisfying the following conditions:
\beqn
\mathcal{J}_{t}\stackrel{d}{\to} \mathcal{J}_{\infty}, \quad \frac{\epsilon(t)}{t} \to 0, \,\, \epsilon(t) \to \infty,  \quad \text{and} \quad \mathcal{J}_{t} - \mathcal{J}_{t-\epsilon(t)} \stackrel{P}{\to} 0,\quad \text{as}  \,\, t \to \infty. \label{key}
\eeqn
 Then, for any $A\in\mathcal{B}(\R^{})$ such that $P[\mathcal{J}_{\infty}\in \partial A]=0$, as $t \to \infty$
\[\lim_{t\to\infty}P\Big[\mathcal{J}_{t} \in A\mid t-\tau_{g_{t}^{j}}^{j}>x , X_{t}=j\Big] =P\Big[\mathcal{J}_{\infty}\in A\Big],\]
for any $x\in\R_{\ge 0}.$
\end{lemma}
Proof of Lemma \ref{lem01} is given in the Appendix \ref{App}.

\subsection{Proof of Lemma \ref{L1}} 
We prove the assertion by showing the following steps, which are proved in following Subsections \ref{S1},\ref{S2} and \ref{S3} subsequently.
\begin{itemize}
\item \textbf{Step $1$:} Under Assumption \ref{As0} for any $t>0$ and $A=(A_{1},A_{2})\in\mathcal{B}(\R^{2})$,
\begin{align*}
P\big[(\Phi^{(\mu)}_{t},I^{(\mu,\lambda)}_{t}) \in A \mid X_{0}=i,X_{t}=j\big]=P\big[\big(Q^{(1)}_{t},Q^{(2)}_{t}\big)\in (A_{1},A_{2}) \mid X_0=i,X_t=j\big].
\end{align*}
where $(Q_{t}^{(1)},Q_{t}^{(2)})$ is defined in \eqref{joint}.
\item \textbf{Step $2$:}  Under Assumptions \ref{As0} and \eqref{condT1} for each $j\in S,$ as $t\to\infty,$ the process $L:=(L_{t}:t\ge 0)$ defined in \eqref{Rt} marginally satisfies $$L_{t}\stackrel{d}{\to}(0,V^{*}_j)$$  where the random variable $V^{*}_j$ is distributed identically as in \eqref{mixture}.
\item \textbf{Step $3$:} Under Assumption \ref{As0} and \eqref{condT1} the process $\big(L_{t}^{(2)}:t\ge 0\big)$ satisfies all properties of $(\mathcal{J}_{t})_{t\ge 0}$ in \eqref{key} of Lemma \ref{lem01}, with $\mathcal{J}_{\infty}=V_{j}^{*}$.
\end{itemize}

We sketch how proofs of the aforementioned steps will show the assertion in \eqref{mixture}. For any $A_{*}\in\mathcal{B}(\R_{\ge 0}\times \R^{}\times S)$ taking $t\to\infty$ limit on both sides of 
\beqn
P_{i}[(\Phi^{(\mu)}_{t},I^{(\mu,\lambda)}_{t},X_{t}) \in A_{*}] = \sum_{j\in S}^{}P[X_{t}=j \mid X_{0}=i]P\big[(\Phi^{(\mu)}_{t},I^{(\mu,\lambda)}_{t},j) \in A_{*} \mid X_{0}=i,X_{t}=j\big],\non
\eeqn
and using $P[X_{t}=j \mid X_{0}=i]\to \pi_{j}$ (which holds under Assumption \ref{As0}), main assertion will be a consequence of Dominated Convergence Theorem if we show that $$P\big[(\Phi^{(\mu)}_{t},I^{(\mu,\lambda)}_{t},X_{t}) \in A_{*} \mid X_{0}=i,X_{t}=j\big]\to P[(0,W_{j},j)\in A_{*}]\s\text{as} \,\, t\to\infty$$
where for each $j\in S,$ the random variable $W_{j}$ is defined in \eqref{mixture}. Since on conditioning with respect to $\{X_{t}=j\}$ last co-ordinate of $(\Phi^{(\mu)}_{t},I^{(\mu,\lambda)}_{t},X_{t})$ becomes deterministic, we proceed with $P\big[(\Phi^{(\mu)}_{t},I^{(\mu,\lambda)}_{t}) \in A \mid X_{0}=i,X_{t}=j\big]$ which has a distributional characterization from \textbf{Step $1$} through $(Q^{(1)}_{t},Q_{t}^{(2)})$. The assertion holds if  we show that as $t\to\infty,$
\beqn
P\Big[(Q^{(1)}_{t},Q^{(2)}_{t})\in A_{} \mid X_{0}=i,X_{t}=j\Big]\to P\Big[(0,W_j)\in A_{}\Big]
\quad A_{}\in \mathcal{B}(\R_{\ge 0}\times \R^{}). \label{toprove2}
\eeqn
Observe that 
\beqn
(Q^{(1)}_{t},Q_{t}^{(2)})=\Big(e^{ -\mu(j)(t-\tau_{g_{t}^{j}}^{j})}L_{t}^{(1)},G_{j}^{\mu,\lambda^{}}\big(t-\tau_{g_{t}^{j}}^{j}\big)+e^{ -\mu(j)(t-\tau_{g_{t}^{j}}^{j})}L_{t}^{(2)}\Big).\label{z_to_r}
\eeqn

The limit in \eqref{toprove2} holds if one can find $t\to\infty$ limit of $\mathcal{L}\Big((t-\tau_{g_{t}^{j}}^{j}), L_{t}\mid X_{0}=i,X_{t}=j\Big).$ From \textbf{Step $2$} one has the unconditional $t\to\infty$ weak limit for $L_{t}.$ As $L_{t}$ is a $\mathcal{F}^{X}_{\tau_{g^{j}_{t}}^{j}}$ measurable process with $L_{t}^{(1)}\stackrel{d}{\to} 0$ established in \textbf{Step $2$}, It suffices to find the limit of $$\mathcal{L}\Big((t-\tau_{g_{t}^{j}}^{j}), L^{(2)}_{t}\mid X_{0}=i,X_{t}=j\Big).$$

Observe that 
\begin{align}
&P\Big[t-\tau_{g_{t}^{j}}^{j}>x, L^{(2)}_{t}\in A\mid X_{0}=i,X_{t}=j\Big]\non\\&\s=P_{i}\big[L^{(2)}_{t}\in A\mid X_{t}=j, t-\tau_{g_{t}^{j}}^{j}>x\big] P_{i}\big[t-\tau_{g_{t}^{j}}^{j}>x\mid X_{t}=j\big].\label{prodt1}
\end{align}
 
Under Assumption \ref{As0}, 
\beqn
\lim_{t\to\infty}P_{i}\big[t-\tau_{g_{t}^{j}}^{j}>x\mid X_{t}=j\big]=\frac{\sum_{k\in S}P_{jk}\int_{x}^{\infty}(1-F_{jk}(y))dy}{m_{j}},\label{eBackward}
\eeqn
that is a consequence of \eqref{lim1} and the limit of the backward residual time $(t-\tau_{g_{t}^{j}}^{j})$ conditioned on $X_{t}=j$ under semi-Markovian setting. The proof of the limit of the backward residual time $(t-\tau_{g_{t}^{j}}^{j})$ follows by similar arguments used in proving the time limit of the forward residual time $\lim_{t\to\infty}P[\tau^{j}_{g_{t}^{j}+1} - t>x, Y_{t}=j]$ at display (8.5) in \cite{cinlar1969markov} and the limit is identical (we omit the proof here).

Observe that the $t\to\infty$ limiting law of $t-\tau_{g_{t}^{j}}^{j}$, conditioned on $\{X_{t}=j\}$ is same as the $\pi_j^*$ in \eqref{pi_j^*}. With the help of \textbf{Step $3$,}  Lemma \ref{lem01} is applied by setting $\mathcal{J}_{t}:=L_{t}^{(2)},\mathcal{J}_{\infty}:=V_{j}^{*},$ and we yield the limit of $P_{i}\big[L^{(2)}_{t}\in A\mid X_{t}=j, (t-\tau_{g_{t}^{j}}^{j})>x\big]$ in \eqref{prodt1}, and for any $A_{1}\in\mathcal{B}(\R), A_{2}\in\mathcal{B}(\R_{\ge 0}),$ this leads to

\[ P_{i}\big[\big(L_{t}^{(2)},(t-\tau_{g_{t}^{j}}^{j})\big)\in A_{1}\times A_{2}\mid X_{t}=j\big]\to P[V^{*}_{j}\in A_{1}] \pi_{j}^{*}(A_{2})\s \text{}\,\, \]

as $t\to\infty$  where $V^{*}_{j}$ is defined as the unconditional limit of $L^{(2)}_{t}$ in \textbf{Step $2$} and the measure $\pi_{j}^{*}(\cdot)$ is defined in \eqref{pi_j^*}. This way \eqref{toprove2} is established by setting the random variable $T_j$ (that is independent of the rest of the random variables) such that $\mathcal{L}(T_j)=\pi^{*}_{j},$ and observing  
\beqn
\mathcal{L}\Big(e^{ -\mu(j)(t-\tau_{g_{t}^{j}}^{j})}L_{t}^{(1)}, G_{j}^{\mu,\lambda^{}}\big(t-\tau_{g_{t}^{j}}^{j}\big)+e^{ -\mu(j)(t-\tau_{g_{t}^{j}}^{j})}L_{t}^{(2)}\mid X_{0}=i, X_{t}=j\Big)\non\\ \stackrel{w}{\to}\mathcal{L}\big(0, G_{j}^{\mu,\lambda^{}}(T_j)+e^{-\mu(j)T_j}V^{*}_{j}\big)
=\mathcal{L}(0,W_{j})\non
\eeqn
as $t\to\infty$ (where $W_{j}$ is defined in \eqref{mixture}). 

\hfill$\square$

\subsection{Proof of \textbf{Step $1$}} \label{S1}

Define $\Phi(s,t)$ as $e^{-\int_{s}^{t} \mu(X_r) \, dr}$ for any $0 < s < t$. On the set $\{ X_0 = i, X_t = j \} = \{ \omega \in \Omega : X_0(\omega) = i, X_t(\omega) = j \}$, we decompose integral expressions involving $(\Phi^{(\mu)}_t, I^{(\mu, \lambda)}_t)$ by partitioning the interval $[0,t)$ into disjoint intervals as
\[
[0, t) = [0, \tau_0^j) \cup \bigcup_{k=1}^{g_t^j} [\tau_{k-1}^j, \tau_k^j) \cup [\tau_{g_t^j}^j, t).
\]
One has the following representation, using the notations in \eqref{LKintro1}, \eqref{LKintro2}, and \eqref{Rt}:
\begin{align}
\begin{pmatrix}
\Phi^{(\mu)}_t \\
I^{(\mu, \lambda)}_t
\end{pmatrix}
&= \begin{pmatrix}
\Phi(0, \tau_0^j) \left( \prod_{i=1}^{g_t^j} \Phi(\tau_{i-1}^j, \tau_i^j) \right) \Phi(\tau_{g_t^j}^j, t) \\
\int_0^{\tau_0^j} \lambda(X_s) \Phi(s,t) \, ds + \sum_{k=1}^{g_t^j} \int_{\tau_{k-1}^j}^{\tau_k^j} \lambda(X_s) \Phi(s,t) \, ds + \int_{\tau_{g_t^j}^j}^{t} \lambda(X_s) \Phi(s,t) \, ds
\end{pmatrix} \nonumber \\
&= \begin{pmatrix}
C_0^j \left( \prod_{k=1}^{g_t^j} C_k^j \right) \Phi(\tau_{g_t^j}^j, t) \\
\Phi(\tau_{g_t^j}^j, t) \left[ D_0^j \prod_{k=1}^{g_t^j} C_k^j + \sum_{k=1}^{g_t^j} \left( \prod_{i=k}^{g_t^j-1} C_{i+1}^j \right) D_k^j \right] + \int_{\tau_{g_t^j}^j}^{t} \lambda(X_s) \Phi(s, t) \, ds
\end{pmatrix} \nonumber \\
&\stackrel{\{X_0 = i, X_t = j\}}{=} \begin{pmatrix}
C_0^j \left( \prod_{k=1}^{g_t^j} C_k^j \right) e^{-\mu(j)(t - \tau_{g_t^j}^j)} \\
e^{-\mu(j)(t - \tau_{g_t^j}^j)} \left[ \left( \prod_{k=1}^{g_t^j} C_k^j \right) D_0^j + \widetilde{P}_{g_t^j}^j \right] + G_j^{\mu, \lambda}(t - \tau_{g_t^j}^j)
\end{pmatrix}\label{}
\end{align}
where the last equality is a consequence of the fact that, on $\{X_0 = i, X_t = j\}$, the latent process $Y$ will always be in state $j$ during the interval $[\tau_{g_t^j}^j, t)$. This implies
\[
\int_{\tau_{g_t^j}^j}^{t} \lambda(X_s) \Phi(s, t) \, ds = \int_{\tau_{g_t^j}^j}^{t} \lambda(j) e^{-\mu(j)(t - s)} \, ds = \int_0^{t - \tau_{g_t^j}^j} \lambda(j) e^{-\mu(j)(t - r)} \, dr = G_j^{\mu, \lambda}(t - \tau_{g_t^j}^j).
\]

Observe that in \eqref{joint}, we have $P_{g_t^j}^j$, which is not the same as $\widetilde{P}_{g_t^j}^j$ in \eqref{PP}. Note that
\begin{align}
&\mathcal{L}\left( (C_0^j, D_0^j), (C_1^j, D_1^j), \ldots, (C_{g_t^j}^j, D_{g_t^j}^j), (t - \tau_{g_t^j}^j) \mid X_0 = i, X_t = j \right) \nonumber \\
&\quad = \mathcal{L}\left( (C_0^j, D_0^j), (C_{g_t^j}^j, D_{g_t^j}^j), \ldots, (C_1^j, D_1^j), (t - \tau_{g_t^j}^j) \mid X_0 = i, X_t = j \right). \label{invar}
\end{align}
This is due to the fact that \( g_{t}^{j} \) is the renewal time, which depends entirely on the total sum of the interval lengths \( [0, \tau_{0}^{j}) \cup_{k=1}^{g_{t}^{j}} [\tau_{k-1}^{j}, \tau_{k}^{j}) \). Reversing the order of the renewal intervals in \( \cup_{k=1}^{g_{t}^{j}} [\tau_{k-1}^{j}, \tau_{k}^{j}) \) by mapping each \( k \) to \( g_{t}^{j} - k + 1 \) for \( k = 1, \ldots, g_{t}^{j} \) leaves this total sum unchanged. Hence, the distributional equality holds on both sides of \eqref{invar}.

Since $P_{g_t^j}^j$ and $\widetilde{P}_{g_t^j}^j$ are outputs of identical functions of the random variables in the LHS and RHS of \eqref{invar}, this implies that
\begin{eqnarray}
\left( (C_0^j, D_0^j), P_{g_t^j}^j, (t - \tau_{g_t^j}^j) \right) \stackrel{\mathcal{L}(\cdot \mid X_0 = i, X_t = j)}{=} \left( (C_0^j, D_0^j), \widetilde{P}_{g_t^j}^j, (t - \tau_{g_t^j}^j) \right), \label{ePP}
\end{eqnarray}
are identical in distribution. Hence,
\begin{align}
&\mathcal{L}\left( C_0^j \left( \prod_{k=1}^{g_t^j} C_k^j \right) e^{-\mu(j)(t - \tau_{g_t^j}^j)}, \, e^{-\mu(j)(t - \tau_{g_t^j}^j)} \left[ \left( \prod_{k=1}^{g_t^j} C_k^j \right) D_0^j + \widetilde{P}_{g_t^j}^j \right] + G_j^{\mu, \lambda}(t - \tau_{g_t^j}^j) \mid X_0 = i, X_t = j \right)\nonumber\\
&\s\s\s\s\s= \mathcal{L}(Q_t^{(1)}, Q_t^{(2)} \mid X_0 = i, X_t = j),\nonumber
\end{align}
proving the assertion of \textbf{Step 1}.
\hfill$\square$

\subsection{Proof of \textbf{Step $2$}}\label{S2}

We prove this step by showing
\begin{itemize}
\item\textbf{Claim $1$:} $L_{t}^{(1)}$ and absolute value of first term of $L_{t}^{(2)}$ both will converge to $0$ in probability as $t\to\infty,$ 
\item \textbf{Claim $2$:} $P_{g_{t}^{j}}^{j}\stackrel{d}{\to}V_{j}^{*}$ in \eqref{mixture} (i.e second term of $L_{t}^{(2)}$ will converge to the solution of SRE $V_{j}^{*}$ in \eqref{mixture}),
\end{itemize}
and as a consequence of Slutsky's theorem assertion of this step holds. 
\subsubsection{Proof of \textbf{Claim $1$}}
To show the first claim, observe that $L_{t}^{(1)}$ and first term of $L_{t}^{(2)}$ are respectively 
\beqn
\bigg(\exp\Big\{t\Big(\frac{\sum_{i=1}^{g_{t}}\log C^{j}_{i}}{g_{t}}+\frac{\log C^{j}_{0}}{g_t}\Big)\frac{g_{t}}{t}\Big\},\exp\Big\{t\Big(\frac{\sum_{i=1}^{g_{t}}\log C^{j}_{i}}{g_{t}}+\frac{\log |D^{j}_{0}|}{g_t}\Big)\frac{g_{t}}{t}\Big\}\bigg).\non
\eeqn
A consequence of renewal theorem, law of large numbers, and Assumption \ref{As0} will suggest that 
\beqn
\frac{\sum_{i=1}^{g_{t}}\log C^{j}_{i}}{g_{t}}\stackrel{P}{\to} E_{\pi}\log C^{j}_{i}=-E|I_{j}|E_{\pi}\mu(\cdot)<0\s \text{and}\s\frac{g^{j}_{t}}{t}\stackrel{a.s}{\to}\frac{1}{E|I_{j}|}.\label{eRenewal}
\eeqn
Since $\log C_{0}^{j}, D_{0}^{j}$ are both $O_{_{P}}(1)$ (follows similar to the display (6.17) of Lemma 6.2 in \cite{lindskog2020exact} using \eqref{condT1}, and a proof of this is given in the supplementary copy of \cite{lindskog2020exact}), both of $\frac{\log C^{j}_{0}}{g^{j}_t}, \frac{\log D^{j}_{0}}{g^{j}_t}$ will go to $0$ in probability. This implies as $t\to\infty,$  $$\bigg(\bigg(\prod_{k=1}^{g_{t}^{j}}C^{j}_{k}\bigg)^{} C^{j}_{0},\bigg(\prod_{k=1}^{g_{t}^{j}}C^{j}_{k}\bigg)^{} D^{j}_{0}\bigg)\stackrel{P}{\to}(0,0).$$ 
\hfill$\square$

\subsubsection{Proof of \textbf{Claim $2$}}\label{lemPerp}
\begin{proof}
We show that \textbf{claim $2$} holds under assumptions of Theorem \ref{T1}, i.e under Assumptions \ref{As0} and \eqref{condT1}, $P_{g_{t}^{j}}^{j}\stackrel{d}{\to}V_{j}^{*}$ in \eqref{mixture} as $t\to\infty$. Define $P^{j}_{\infty}:=\sum_{k=1}^{\infty}\bigg(\prod_{i=1}^{k-1}C_{i}^{j}\bigg)D_{k}^{j},$ and we show that $P^{j}_{g_{t}^{j}}\stackrel{P}{\to}P^{j}_{\infty}$ and $P^{j}_{\infty}$ satisfies unique solution of SRE that is identical as $\mathcal{L}(V_{j}^{*}).$ Note that
\beqn
P^{j}_{\infty}=\sum_{k=1}^{\infty}\bigg(\prod_{i=1}^{k-1}C_{i}^{j}\bigg)D_{k}^{j}=D_{1}^{j}+C_{1}^{j}\Big(\sum_{k=2}^{\infty}\bigg(\prod_{i=2}^{k-1}C_{i}^{j}\bigg)D_{k}^{j}\Big).\non
\eeqn
Now since $(C_{1}^{j},D_{1}^{j})\perp \{(C_{i}^{j},D_{i}^{j}): i\ge 2\},$ it follows that the second term inside first bracket of RHS $\sum_{k=2}^{\infty}\bigg(\prod_{i=2}^{k-1}C_{i}^{j}\bigg)D_{k}^{j}\stackrel{d}{=} P^{j}_{\infty} $ and also $(C_{1}^{j},D_{1}^{j})\perp \sum_{k=2}^{\infty}\bigg(\prod_{i=2}^{k-1}C_{i}^{j}\bigg)D_{k}^{j},$ it follows that $P^{j}_{\infty}$ satisfies the SRE at \eqref{mixture}. Under \eqref{condT1}, we have $$E\log C_{1}^{j}=-E|I_{j}|E_{\pi}\mu(\cdot)<0,\s\text{ and }\s E\log^{+}D_{1}^{j}<\infty,$$ it follows that the solution to the SRE $X\stackrel{d}{=}D_{1}^{j}+C_{1}^{j}X$ is unique in distribution (Lemma 1.4(a) and Vervaat's Theorem 1.5 in \cite{vervaat1979stochastic}). So it follows that $\mathcal{L}(P_{\infty}^{j})=\mathcal{L}(V_{j}^{*})$ in \eqref{mixture}.

Observe that for any $t>0$,  
\beqn
P^{j}_{\infty}-P^{j}_{g_{t}^{j}}&\stackrel{}{=}& \Big(\prod_{i=1}^{g_{t}^{j}}C_{i}^{j}\Big)\bigg[\sum_{k=g_{t}^{j}+1}^{\infty}\Big(\prod_{i=g_{t}^{j}+1}^{k-1}C_{i}^{j}\Big)D^{j}_{k}\bigg]\non\\
&=& \exp\Big\{t\Big(\frac{\sum_{i=1}^{g_{t}}\log C^{j}_{i}}{g_{t}}\Big)\frac{g_{t}^{j}}{t}\Big\} \bigg[\sum_{k=g_{t}^{j}+1}^{\infty}\bigg(\prod_{i=g_{t}^{j}+1}^{k-1}C_{i}^{j}\Big)D^{j}_{k}\bigg].\non
\eeqn
As $t\to\infty,$ using \eqref{eRenewal} first term $\prod_{i=1}^{g_{t}^{j}}C_{i}^{j}\stackrel{P}{\to}0.$ To prove the assertion, we prove that the second term (in the RHS above) is $O_{P}(1).$ Observe that for any $t\ge 0,$
\beqn
\big((C^{j}_{g^{j}_{t}+1},D^{j}_{g^{j}_{t}+1}),(C^{j}_{g^{j}_{t}+2},D^{j}_{g^{j}_{t}+2}),\ldots \big)\stackrel{d}{=}\big((C^{j}_{1},D^{j}_{1}),(C^{j}_{2},D^{j}_{2}),\ldots\big)\label{s'2}
\eeqn
which follows by conditioning w.r.t $g^{j}_{t}$ and using argument similar to \eqref{eres2}. A consequence of \eqref{s'2} gives 
\beqn
P^{j}_{\infty}\stackrel{d}{=} \sum_{k=g_{t}^{j}+1}^{\infty}\bigg(\prod_{i=g_{t}^{j}+1}^{k-1}C_{i}^{j}\Big)D^{j}_{k}=D^{j}_{g_{t}+1}+C^{j}_{g_{t}+1}\big(D^{j}_{g_{t}+2}+C^{j}_{g_{t}+2}\big(\ldots\big)\big),\label{s2'}
\eeqn
and since distribution of $P^{j}_{\infty}$ uniquely exists under \eqref{condT1}, the random quantity $\sum_{k=g_{t}^{j}+1}^{\infty}\bigg(\prod_{i=g_{t}^{j}+1}^{k-1}C_{i}^{j}\Big)D^{j}_{k}
$ is $O_{_P}(1),$ proving our assertion. 
\hfill$\square$
\end{proof}

\subsection{Proof of \textbf{Step $3$}}\label{S3}

Since the first term of $L^{(2)}_{t}$ i.e $\big(\prod_{k=1}^{g_{t}^{j}}C^{j}_{k}\big)^{} D^{j}_{0}\stackrel{P}{\to}0$ as $t\to\infty.$ The assertion holds if we show \eqref{key} of Lemma \ref{lem01}  with $P_{g_{t}^{j}}^{j}$ as $\mathcal{J}_{t}.$ Note that from \textbf{Step $2$} it follows that unconditionally $P_{g_{t}^{j}}^{j}\stackrel{d}{\to}V^{*}_{j},$ so we can take $V^{*}_{j}$ as $\mathcal{J}_{\infty}$  in \eqref{key}. We are only left to prove $P_{g_{t}^{j}}^{j} - P_{g_{t-\epsilon(t)}^{j}}^{j}\stackrel{P}{\to} 0$ for some function $\epsilon(\cdot)$ such that $\epsilon(t)\to\infty$ and $\frac{\epsilon(t)}{t}\to 0$ as $t\to\infty.$

\beqn
P_{g_{t}^{j}}^{j}- P_{g_{t-\epsilon(t)}^{j}}^{j}&=&\sum_{k=1}^{g_{t}^{j}}\Big(\prod_{i=1}^{k-1}C_{i}^{j}\Big)D_{k}^{j}- \sum_{k=1}^{g_{t-\epsilon(t)}^{j}}\Big(\prod_{i=1}^{k-1}C_{i}^{j}\Big)D_{k}^{j}\non\\
&=&\bigg(\prod_{i=1}^{g_{t-\epsilon(t)}^{j}}C_{i}^{j}\bigg)\bigg[\sum_{k=g^{j}_{t-\epsilon(t)}+1}^{g^{j}_{t}}\bigg(\prod_{i=g_{t-\epsilon(t)}^{j}+1}^{k-1}C_{i}^{j}\bigg)D_{k}^{j}\bigg]\label{st3e1}
\eeqn
Observe that the second term of the product in \eqref{st3e1} can be bounded by 
\beqn
\sum_{k=g^{j}_{t-\epsilon(t)}+1}^{g^{j}_{t}}\bigg(\prod_{i=g_{t-\epsilon(t)}^{j}+1}^{k-1}C_{i}^{j}\bigg)D_{k}^{j}&\le& \sum_{k=g^{j}_{t-\epsilon(t)}+1}^{\infty}\bigg(\prod_{i=g_{t-\epsilon(t)}^{j}+1}^{k-1}C_{i}^{j}\bigg)D_{k}^{j}.\non\\&\stackrel{d}{=}&\sum_{k=1}^{\infty}\bigg(\prod_{i=1}^{k-1}C_{i}^{j}\bigg)D_{k}^{j}\non
\eeqn
using ideas similar to \eqref{s'2} and \eqref{s2'}) where the upper bound satisfies the unique solution of the SRE $$X=C_{1}^{j}X+D_{1}^{j},\s X\perp (C_{1}^{j},D_{1}^{j})$$
under conditions $E\log C_{k}^{j}=-E|I_{j}|E_{\pi}\mu(\cdot)<0,\,\, E\log^{+}D_{k}^{j}<\infty$  ensured by \eqref{condT1}. This implies that the second term of the product in \eqref{st3e1} is $O_{P}(1)$.

Observe that first term in the product of \eqref{st3e1} is $$\bigg(\prod_{i=1}^{g_{t-\epsilon(t)}^{j}}C_{i}^{j}\bigg)= \exp\Big\{(t-\epsilon(t))\Big(\frac{\sum_{i=1}^{g^{j}_{t-\epsilon(t)}}\log C^{j}_{i}}{g^{j}_{t-\epsilon(t)}}\Big)\frac{g_{t-\epsilon(t)}^{j}}{t-\epsilon(t)}\Big\}.$$
A consequence of Renewal Theorem and Law of large number leads to \eqref{eRenewal} with $t$ replaced by $t-\epsilon(t)$ and as $t-\epsilon(t) \to\infty,$ one has $\prod_{i=1}^{g_{t-\epsilon(t)}^{j}}C_{i}^{j}\stackrel{P}{\to}0$

\hfill$\square$

\section{Appendix}\label{App}
We prove the Proposition \ref{Psemi1} and Lemma \ref{lem01} here.

\subsection{Proof of Proposition \ref{Psemi1}}
\begin{proof}
 Observe that
\beqn
P[\tau_{g^{j}_{t}+2}^{j}-\tau_{g^{j}_{t}+1}^{j}>x]&=&\sum_{n=0}^{\infty}P[\tau_{g^{j}_{t}+2}^{j}-\tau_{g^{j}_{t}+1}^{j}>x\mid g^{n}_{t}=n]P[g^{n}_{t}=n]\non\\
&=&\sum_{n=0}^{\infty}P[\tau_{n+2}^{j}-\tau_{n+1}^{j}>x\mid g^{j}_{t}=n]P[g^{j}_{t}=n]\non\\
&\stackrel{(a)}{=}&\sum_{n=0}^{\infty}P[\tau_{1}^{j}-\tau_{0}^{j}>x]P[g^{j}_{t}=n]=P[\tau_{1}^{j}-\tau_{0}^{j}>x]\label{eres2}
\eeqn
where equality in (a) holds from the fact that the event $\{g_{t}^{j}=n\}=\{\tau_{n}^{j}\le t <\tau_{n+1}^{j}\}$ is independent with the event $\{\tau_{n+2}^{j}-\tau_{n+1}^{j}>x\}$  due to the regeneration property, and the latter is identical in probability with $\{\tau_{1}^{j}-\tau_{0}^{j}>x\}$. Similar to steps done in \eqref{eres2} one can show that any functional of $X$ in $\{s: s\ge \tau^{j}_{g_{t}^{j}+1}\}$ is independent of $g_{t}^{j}$ and identically distrubuted as that functional on $\{s:s\ge \tau_{0}^{j}\}$ for any $t\ge \tau_{0}^{j}$ (proving the assertion of part (ii)).

Proof of part (iii) begins by observing that $\{g^{j}_{t}=n, X_{t}=j'\}=\{\tau^{j}_{n}\le t <\tau^{j}_{n+1},X_{t}=j'\}.$ For any two sets $A,B$
\beqn
&&P[X^{(1)}_{t}\in A ,X^{(2)}_{t} \in B\mid K_{t},X_{t}=j']\non\\&\s\s=&\sum_{n=0}^{\infty}P\big[X^{(1)}_{t}\in A ,X^{(2)}_{t} \in B\mid K_{t}, X_{t}=j',g^{j}_{t}=n\big]P[g^{j}_{t}=n\mid K_{t},X_{t}=j']\non\\
&\s\s\stackrel{(c)}{=}& \sum_{n=0}^{\infty}P\big[X^{(1)}_{t}\in A\mid  K_{t},g^{j}_{t}=n, X_{t}=j'\big] P\big[X^{(2)}_{t} \in B\mid  K_{t}, g^{j}_{t}=n,  X_{t}=j'\big]\non\\&&\s\times P[g^{j}_{t}=n\mid  K_{t}, X_{t}=j']\non\\
&\s\s\stackrel{(d)}{=}& \sum_{n=0}^{\infty}P\big[X^{(1)}_{t}\in A\mid  K_{t}, g^{j}_{t}=n, X_{t}=j'\big] P\big[X^{(2)}_{t} \in B\mid  K_{t}\big]P[g^{j}_{t}=n\mid  K_{t}, X_{t}=j']\non\\
&\s\s=& P[X^{(1)}_{t}\in A\mid  K_{t}, X_{t}=j'\big]P\big[X^{(2)}_{t} \in B \mid K_{t}].\non
\eeqn
Equality in $(c)$ follows from the fact that conditioned on $\{K_{t},g^{j}_{t}=n,X_{t}=j'\}$ the processes $X^{(1)}_{t},X^{(2)}_t$ are respectively $\sigma\{\mathcal{H}_{i}:i\le n\}$ and $\sigma\{\mathcal{H}_{i}: i\ge n+2\}$ measurable which are independent sigma fields. Equality in $(d)$ follows as conditioned on $\{K_{t}\}$ and $\{g_{t}^{j}=n,X_{t}=j'\};$ $X_{t}^{(2)}$ is $\sigma\{\mathcal{H}_{i}:i\ge n+2\}$ measurable and independent with $\{g_{t}^{j}=n,X_{t}=j'\}=\{\tau_{n}^{j}\le t<\tau_{n+1}^{j},X_{t}=j'\}$ as it is $\sigma\{\mathcal{H}_{i}: i\le n+1\}$ measurable.
\hfill$\square$
\end{proof}

\subsection{Proof of Lemma \ref{lem01}}\label{lem2}
Proof of Lemma \ref{lem01} follows similarly by \textbf{Step $2$} and \textbf{Step $3$} of Lemma 6.1 of \cite{majumder2024long}. For completeness, we give a proof here.

\begin{proof}
We prove the assertion in the following two steps. 
\begin{itemize}
\item \textbf{Step $1$:} $\lim_{t\to\infty}P\Big[\mathcal{J}_{t-\epsilon(t)}\in A\,\mid\,t-\tau_{g_{t}^{j}}^{j}>x, X_{t}=j\Big]=P[\mathcal{J}_{\infty}\in A].$
\item \textbf{Step $2$:} $\lim_{t\to\infty}P\Big[\mathcal{J}_{t}\in A\,\mid\,t-\tau_{g_{t}^{j}}^{j}>x, X_{t}=j\Big]=\lim_{t\to\infty}P\Big[\mathcal{J}_{t-\epsilon(t)}\in A\,\mid\,t-\tau_{g_{t}^{j}}^{j}>x, X_{t}=j\Big]$ for any $A$ such that $P\big[\mathcal{J}_{\infty}\in\partial A\big]=0.$
\end{itemize}

\textbf{Step 1:} Denote the event $\{X_{t}=j,t-\tau^{j}_{g_{t}^{j}}>x\}$ by $\widetilde{G}^{(j)}_{t,x}.$ Observe that for any set $A\in\mathcal{B}(\R),$
\begin{eqnarray}
P\Big[\mathcal{J}_{t-\epsilon(t)}\in A\,\mid\,\widetilde{G}^{(j)}_{t,x}\Big] &=&\sum_{j'\in S}P\Big[\mathcal{J}_{t-\epsilon(t)}\in A\,\mid\, X_{t-\epsilon(t)}=j',\widetilde{G}^{(j)}_{t,x}\Big]P\Big[X_{t-\epsilon(t)}=j'\mid \widetilde{G}^{(j)}_{t,x}\Big]\non\\&=&\sum_{j'\in S}\frac{P\big[\widetilde{G}^{(j)}_{t,x}\mid X_{t-\epsilon(t)}=j', \mathcal{J}_{t-\epsilon(t)}\in A\big]}{P\big[\widetilde{G}^{(j)}_{t,x}\big]}P\big[\mathcal{J}_{t-\epsilon(t)}\in A, X_{t-\epsilon(t)}=j'\big]\non\\
&=& P[\mathcal{J}_{t-\epsilon(t)}\in A]+ \sum_{j'\in S}P\Big[\mathcal{J}_{t-\epsilon(t)}\in A,\, X_{t-\epsilon(t)}=j'\Big]\label{eqSt2.10}\\ &&\s\s\times\Bigg(\frac{P[\widetilde{G}^{(j)}_{t,x}\mid \mathcal{J}_{t-\epsilon(t)}\in A, X_{t-\epsilon(t)}=j' ]}{P[\widetilde{G}^{(j)}_{t,x}]}-1\Bigg).\non
\end{eqnarray}

Denote the second summation term in the RHS of \eqref{eqSt2.10} by $\mathcal{D}_{t}$ and the quantity in first bracket by $$M_{t}:=\frac{P[t-\tau_{g_{t}^{j}}^{j}>x, X_{t}=j\mid \mathcal{J}_{t-\epsilon(t)}\in A, X_{t-\epsilon(t)}=j' ]}{P[t-\tau_{g_{t}^{j}}^{j}>x, X_{t}=j]}.$$ In order to show $\mathcal{D}_{t}\to 0,$ we first prove that $\lim_{t\to\infty}M_{t}\to 1.$  

Denote the set $\{\tau_{g_{t-\epsilon(t)}^{j}+2}^{j}\le \tau_{g_{t}^{j}}^{j}\}$ by $K_{t}.$ Observe that 

$$K_{t}=\Big\{B_{j}(t-\epsilon(t))+(\tau_{g^{j}_{t-\epsilon(t)}+2}^{j}-\tau_{g^{j}_{t-\epsilon(t)}+1}^{j})+A_{j}(t)\le \epsilon(t)\Big\}.$$ 
It follows that $P[K_t] \to 1$ as $t \to \infty$ (since  $\epsilon(t) \to \infty$ and all the terms $B_{j}(t-\epsilon(t))$,$A_{j}(t)$ and $\tau_{g^{j}_{t-\epsilon(t)}+2}^{j}-\tau_{g^{j}_{t-\epsilon(t)}+1}^{j}$ are $O_P(1)$ using \eqref{eRes} and Proposition \ref{Psemi1}(a) ).

Denote the terms 
\beqn
\widetilde{a}_{t}&:=& P\big[K_{t}\mid X_{t-\epsilon(t)}=j', \mathcal{J}_{t-\epsilon(t)}\in A\big],\non\\
\widetilde{b}_{t}&:=& P\big[K_{t}^{c}, (t-\tau_{g_{t}^{j}}^{j})>x, X_{t}=j\mid \mathcal{J}_{t-\epsilon(t)}\in A, X_{t-\epsilon(t)}=j'\big]\non
\eeqn

 and it is easy to see  that as $t\to\infty,$ $$(\widetilde{a}_{t},\widetilde{b}_{t})\to (1,0).$$ Observe that the numerator of $M_{t}$
\begin{align}
&P\big[t-\tau_{g_{t}^{j}}^{j}>x, X_{t}=j\mid \mathcal{J}_{t-\epsilon(t)}\in A, X_{t-\epsilon(t)}=j'\big]\non\\&\s=\,\,\,\,\,\,\widetilde{a}_{t}\,\,P\big[t-\tau_{g_{t}^{j}}^{j}>x, X_{t}=j\mid K_{t}, \mathcal{J}_{t-\epsilon(t)}\in A, X_{t-\epsilon(t)}=j'\big]+\widetilde{b}_{t}\non\\&\s\stackrel{(a)}{=}\,\,\,\,\,\,\,\,\widetilde{a}_{t}\,\, P\big[t-\tau_{g_{t}^{j}}^{j}>x, X_{t}=j\mid  K_{t}\big] +\widetilde{b}_{t}\,\,  ,\label{eq_Mt1}
\end{align}
where equality in $(a)$ follows by observing that, conditioned on \( K_{t} \) and \( \{X_{t-\epsilon(t)}=j'\} \), the event \( \{t-\tau^{j}_{g_{t}^{j}}>x, Y_{t}=j\} \) is measurable with respect to \( \sigma\{\mathcal{H}_{i}: i \geq g_{t-\epsilon(t)}^{j}+2\} \), while \( \mathcal{J}_{t-\epsilon(t)} \) is measurable with respect to \( \sigma\{\mathcal{H}_{i}: i \leq g_{t-\epsilon(t)}^{j}\} \). Therefore, by applying Proposition \ref{Psemi1}(c), we can omit \( \{\mathcal{J}_{t-\epsilon(t)} \in A\} \) from the conditioning part. Moreover, we can remove \( \{X_{t-\epsilon(t)}=j'\} \) from the conditioning part as well,  since it is measurable with respect to \( \mathcal{F}^{X}_{\tau_{g^{j}_{t-\epsilon(t)}+1}^{j}} \), which is independent of \( \sigma\{\mathcal{H}_{i}: i \geq g_{t-\epsilon(t)}^{j}+2\} \), under which \( \{t-\tau^{j}_{g_{t}^{j}}>x, X_{t}=j\} \) is measurable when conditioned on \( K_{t} \).

 Since $P[K_{t}]\to 1\,\, \&\,\, (\widetilde{a}_{t},\widetilde{b}_{t})\to (1,0)$, both the numerator and denominator of $M_{t}$ will have an identical time limit, which implies that $\lim_{t\to\infty}M_{t}\to 1.$

Using dominated convergence theorem (DCT) type argument, now we will prove that $D_{t}\to0$ as $t\to\infty.$

Denote $\frac{\mu_{j}\sum_{k\in S}P_{jk}\int_{x}^{\infty}(1-F_{jk}(y))dy}{\sum_{k\in S}\mu_{k}m_{k}}$ by $\pi_{j,x},$ which is the limit of $P\big[t-\tau_{g_{t}^{j}}^{j}>x, X_{t}=j\big]$ (follows from \eqref{eBackward}).

 For a fixed $j\in S, x\ge 0,$ and any $\delta_{x} \in (0,\pi_{j,x})$  there exists $t_{\delta_{x}}>0$ such that 
\beqn
\Big|P\big[t-\tau_{g_{t}^{j}}^{j}>x, X_{t}=j\big]-\pi_{j,x}\Big|\leq \delta_{x},\s t\ge t_{\delta_{x}}.\label{estep2}
\eeqn
 Observe that each term inside the sum of the second term of the RHS of \eqref{eqSt2.10} converges to $0$ as $t\to\infty$. Also note that, for $t\ge t_{\delta_{x}}$, each summand has an upper bound
\begin{align*}
P\big[\mathcal{J}_{t-\varepsilon(t)}\in A, X_{t-\varepsilon(t)}=j'\big]\big|M_{t}-1\big| 
&\le P\big[X_{t-\varepsilon(t)}=j'\big]\frac{2}{P\big[t-\tau_{g_{t}^{j}}^{j}>x, X_{t}=j\big]} \\
&\le P\big[X_{t-\varepsilon(t)}=j'\big]\frac{2}{\pi_{j,x}-\delta_{x}}
\end{align*} 
and that the upper bound is summable over $j'\in S$ and has a limit as $t\to\infty$ which is also summable over $j'\in S$. Hence, by Pratt's lemma (p.~101 in Schilling \cite{Schilling05}), the second term in the RHS of \eqref{eqSt2.10} converges to $D_{t}\to0$ as $t\to\infty$.

From the assumption \( \frac{\epsilon(t)}{t} \to 0 \), the assertion \( \mathcal{J}_{t-\epsilon(t)} \stackrel{d}{\to} \mathcal{J}_{\infty} \) follows, which shows that the first term in \eqref{eqSt2.10} converges to \( P[\mathcal{J}_{\infty} \in A] \), as asserted.

\textbf{Step 2:} Take any $\delta>0$ and define $A^{\delta}:=\{x: d(x,A)<\delta\},A^{-\delta}:=\{x\in A: d(x,A^c)> \delta\}$ and note that $A^{-\delta}\subseteq A\subseteq A^{\delta}$. Denoting $B_{t}:=\mathcal{J}_{t}-\mathcal{J}_{t-\epsilon(t)}$ one has 
\begin{eqnarray}
P\big[\mathcal{J}_{t}\in A \,\mid\,t-\tau_{g_{t}^{j}}^{j}>x, X_{t}=j\big]=P\big[\mathcal{J}_{t-\epsilon(t)}+B_{t}\in A, |B_{t}|\le \delta \,\mid\,t-\tau_{g_{t}^{j}}^{j}>x, X_{t}=j\big]+C_{t}\non\label{cond0.03}
\end{eqnarray}
where $$ \limsup_{t\to\infty}C_{t}\le \limsup_{t\to\infty}\frac{P[|B_{t}|>\delta]}{P\big[t-\tau_{g_{t}^{j}}^{j}>x, X_{t}=j\big]}=0.$$  
Observe that
\beqn
P\big[\mathcal{J}_{t-\epsilon(t)}\in A^{-\delta}\mid \,\,t-\tau_{g_{t}^{j}}^{j}>x, X_{t}=j\big]&\le& P\big[\mathcal{J}_{t-\epsilon(t)}+B_{t}\in A, |B_{t}|\le \delta\mid \,\,t-\tau_{g_{t}^{j}}^{j}>x, X_{t}=j\big]\non\\
&\le& P\big[\mathcal{J}_{t-\epsilon(t)}\in A^{\delta}\mid \,\,t-\tau_{g_{t}^{j}}^{j}>x, X_{t}=j\big].\non
\eeqn
Step $2$ follows by letting $\delta\to 0$ after taking $t\to\infty$ to both sides of the above step and applying the result obtained from Step 1 (using continuity set assumption of $A$ i.e $P[\mathcal{J}_{\infty}\in \partial A]=0$). 
\hfill$\square$
\end{proof}

\end{document}